\documentclass[11pt]{article}
\usepackage{titlesec}
\usepackage{enumerate}
\titleformat{\section}{\centering\normalsize}{\textbf{\thesection .}}{1em}{\textbf}
\usepackage{amsmath, amsthm, amsfonts, amssymb}
\usepackage{textcomp}
\usepackage{color}

\theoremstyle{plain}
\newtheorem{thm}{Theorem}[section]
\newtheorem{lem}[thm]{Lemma}
\newtheorem{prop}[thm]{Proposition}
\newtheorem{cor}[thm]{Corollary}
\theoremstyle{definition}
\newtheorem{rmk}[thm]{Remark}
\newtheorem{ex}[thm]{Example}
\newtheorem{defi}[thm]{Definition}

\usepackage[text={16cm,23cm},centering]{geometry}
\numberwithin{equation}{section}

\newcommand{\R}{\mathbb{R}}

\newcommand{\emb}{\hookrightarrow}
\newcommand{\eps}{\varepsilon}
\newcommand{\norm}[1]{\| #1 \|}
\newcommand{\abs}[1]{| #1 |}
\newcommand{\Z}{\mathbb{Z}}
\newcommand{\N}{\mathbb{N}}
\newcommand{\No}{\mathbb{N}_0}
\newcommand{\C}{\mathbb{C}}
\newcommand{\Sc}{\mathcal{S}}
\newcommand{\supp}{\text{supp }}
\renewcommand{\supp}{\mathop{\mathrm{supp}}}
\newcommand{\bl}[1]{\boldsymbol\ell_{\bf{#1}}}
\newcommand{\tr}{\text{\normalfont{tr}}_{\R^{n-1}}}
\newcommand{\trk}{\text{\normalfont{tr}}_{\R^{n-k}}}
\newcommand{\ext}{\text{ext }}
\renewcommand{\ext}{\mathop{\mathrm{ext}}}
\newcommand{\id}{\mathop{\mathrm{id}}}

\newcommand{\blue}[1]{{\color{blue}#1}}
\renewcommand{\blue}[1]{{#1}}
\newcommand{\magenta}[1]{{\color{magenta}#1}}
\renewcommand{\magenta}[1]{{#1}}
\newcommand{\dint}{\,\mathrm{d}}
\newcommand{\Dd}{{\mathrm{D}}}
\newcommand{\ignore}[1]{}

\begin{document}

\title{Traces of some weighted function spaces and related non-standard real interpolation of Besov spaces}
\author{Blanca F. Besoy\footnote{Supported by MTM2017-84058-P(AEI/FEDER, UE) and FPU grant FPU 16/02420 of the Spanish Ministerio de Educaci\'on, Cultura y Deporte.} \and Dorothee D. Haroske \and Hans Triebel}
\AtEndDocument{\bigskip{\footnotesize%
  \textbf{Blanca F. Besoy }\textsc{Departamento de Ingeniería Geológica y Minera, Escuela Técnica Superior de Ingenieros de Minas y Energía, Universidad Politécnica de Madrid. Calle de Alenza, 4, 28003, Madrid.} \par  
  \textit{E-mail address:} \texttt{blanca.fbesoy@upm.es} \par
  \addvspace{\medskipamount}
  \textbf{Dorothee D. Haroske }\textsc{Institute of Mathematics, Friedrich Schiller University Jena, 07737 Jena, Germany.} \par
  \textit{E-mail address:} \texttt{dorothee.haroske@uni-jena.de} \par
  \addvspace{\medskipamount}
  \textbf{Hans Triebel }\textsc{Institute of Mathematics, Friedrich Schiller University Jena, 07737 Jena, Germany.} \par
  \textit{E-mail address:} \texttt{hans.triebel@uni-jena.de}
}}

\date{}

\maketitle

\begin{abstract}
\noindent
We study traces of weighted Triebel-Lizorkin spaces $F^s_{p,q}(\R^n,w)$ on hyperplanes $\R^{n-k}$, where the weight is of Muckenhoupt type. We concentrate on the example weight $w_\alpha(x) = |x_n|^\alpha$ when $|x_n|\leq 1$, $x\in\R^n$, and $w_\alpha(x)=1$ otherwise, where $\alpha>-1$. Here we use some refined atomic decomposition argument as well as an appropriate wavelet representation in corresponding (unweighted) Besov spaces.  The second main outcome is the description of the real interpolation space $(B^{s_1}_{p_1,p_1}(\R^{n-k}), B^{s_2}_{p_2,p_2}(\R^{n-k}))_{\theta,r}$, $0<p_1<p_2<\infty$, $s_i=s-(\alpha+k)/{p_i}$, $i=1,2$, $s>0$ sufficiently large, $0<\theta<1$, $0<r\leq\infty$. Apart from the case $1/r= (1-\theta)/{p_1}+ {\theta}/{p_2}$ the question seems to be open for many years. Based on our first result we can now quickly solve this long-standing problem. Here we benefit from some very recent finding of Besoy, Cobos and Triebel.\\



\end{abstract}

\section{Introduction}
We study functions which belong to some weighted spaces of Besov and Triebel--Lizorkin
type, $B^s_{p,q}(\R^n,w)$ and $F^s_{p,q}(\R^n,w)$, thus including Sobolev spaces, where the weight function belongs to some
Muckenhoupt class. Such spaces have been treated systematically by
{Bui et al.} in \cite{Bui82,Bui84,Bui96,Bui97}. Later this topic was revived and extended by {Rychkov} in \cite{rychkov-5}, including also approaches for locally regular weights.

Our main attention here is to determine the traces of such weighted spaces on hyperplanes. 
%
Trace questions are of particular interest in view of boundary value problems
of elliptic operators, where some singular behaviour near the boundary
(characterised by the appropriate Muckenhoupt weight) may occur. A standard approach is to start with assertions about traces on hyperplanes and then to transfer these findings to spaces defined on bounded domains with sufficiently smooth boundary. Further studies may concern compactness or regularity results,
leading to the investigation of spectral properties. First partial results can be found
in \cite{nik,T78} for domains $\Omega$ with smooth boundaries
$\partial\Omega$ and Muckenhoupt  weights of type
$w(x)=\left(\mathrm{dist}(x,\partial\Omega)\right)^\gamma$,  $\gamma>-1$. This was
further extended to fractal $d$-sets $\Gamma$ in \cite{iwona-phd}, using the
atomic approach \cite{HP} and based on ideas for the unweighted case in
\cite{T97}. Parallel observations for special weights can be found in \cite{HM,HaSchm}. 

The example weight we are interested in here is the weight $w_\alpha(x)$, $\alpha>-1$, defined by $w_\alpha(x)=|x_n|^\alpha$ when $|x_n|\leq 1$, and $w_\alpha(x)=1$ otherwise. As it is well-known, $w_\alpha $ belongs to the largest Muckenhoupt class $\mathcal{A}_\infty$ when $\alpha>-1$. We can prove that, appropriately interpreted,
\[
\trk F^s_{p,q}(\R^{n}, w_\alpha) = B^{s-\frac{\alpha+k}{p}}_{p,p}(\R^{n-k}),
\]
where $0<p<\infty$, $0<q\leq\infty$, $k\in \{1, \dots, n-1\}$, $\alpha>-1$, and $s-\frac{\alpha+k}{p}>(n-k) \max\{\frac1p-1,0\}$. In particular, we construct  a linear and bounded extension operator $\ext$ from $B^{s-\frac{\alpha+k}{p}}_{p,p}(\R^{n-k})$ to $F^s_{p,q}(\R^{n}, w_\alpha)$ such that $\trk \circ \ext$ is the identity in $B^{s-\frac{\alpha+k}{p}}_{p,p}(\R^{n-k})$. 
We refer to Sections~\ref{prelim} and \ref{F-weight} for the details and definitions. This result extends previous findings in \cite{P,HM}. We prove it based on an atomic decomposition result \cite{HP} in combination with the wavelet decomposition of (unweighted) Besov spaces in \cite[Theorem 1.20]{T08}. In that way everything is shifted to the argument on the sequence space side.

Our second main goal is the answer to a long-standing problem in real interpolation theory, mentioned already by Peetre in \cite[p.110]{Peetre}. It is well known, that the real interpolation space 
\[
(B_{p_1, p_1}^{s-\frac{1}{p_1}}(\R^{n}), B_{p_2, p_2}^{s-\frac{1}{p_2}}(\R^{n}))_{\theta,p} = B_{p,p}^{s-\frac{1}{p}}(\R^{n}),
\]
where   $0<p_1<p_2<\infty$, $0<\theta<1$, $\frac1p= \frac{1-\theta}{p_1} + \frac{\theta}{p_2}$, and $s-\frac{1}{p_j}>n \max\{\frac{1}{p_j}-1,0\}$, $j=1,2$.  But what is the resulting space in the more general situation 
\[(B_{p_1,p_1}^{s-\frac{\alpha}{p_1}}(\R^{n}), B_{p_2,p_2}^{s-\frac{\alpha}{p_2}}(\R^{n}))_{\theta,r} ,\] 
where $\alpha>0$, and $0<r\leq\infty$, but $r \neq p$? It seems that no answer has been obtained so far. Based on our results in Section~\ref{F-weight} we can now easily describe the resulting space via its wavelet representation. Here we benefit from the very recent outcome in \cite{BCT}. In the very end we present an alternative argument which works for all $\alpha\in\R$.

The paper is organised as follows. In Section~\ref{prelim} we collect some
notation, and the basic facts about Muckenhoupt weights and function spaces of Besov and Triebel-Lizorkin type, as far as needed in the sequel. In Section~\ref{F-weight} we concentrate on the trace space $\trk F^s_{p,q}(\R^n,w_\alpha)$, whereas Section~\ref{interpol} is devoted to the new result in real interpolation of Besov spaces.

\section{Weights and function spaces}\label{prelim}
We start with a brief introduction of the \emph{Muckenhoupt classes} $\mathcal{A}_{p}$. For a Lebesgue measurable set $E \subset \R^{n}$, we denote by $\abs{E}$ the Lebesgue measure of the set $E$ on $\R^{n}$ and by a weight $w$ we shall always mean a locally integrable function $ w \in L_{1}^{\text{loc}}(\R^{n})$ and positive almost everywhere. Let $M$ stand for the \emph{Hardy-Littlewood maximal operator} given by
\begin{equation}\label{eq:maximal operator}
Mf(x) = \sup_{r > 0} \frac{1}{\abs{B(x,r)}} \int_{B(x,r)} \abs{f(y)} \dint y, \quad x \in \R^{n},
\end{equation}
where $B(x,r) = \{ y \in \R^{n}: \abs{y-x} < r\}$ and $\abs{B(x,r)}$ denotes the Lebesgue measure of the ball $B(x,r)$. 
\begin{defi}\label{def:Muckenhoupt class}
Let $w$ be a weight on $\R^{n}$. 
\begin{enumerate}[\bfseries (i)]
\item We say that $w$ belongs to the Muckenhoupt class $\mathcal{A}_{p}$, $ 1 < p < \infty$, if there exists a constant $ C > 0$ such that for all balls $B$ the following inequality holds
\[
\Big(\frac{1}{\abs{B}} \int_{B} w(x) \dint x \Big)^{1/p} \Big(\frac{1}{\abs{B}} \int_{B} w(x)^{-p'/p} \magenta{\dint x}\Big)^{1/p'}\leq C,
\]
where $\frac{1}{p}+\frac{1}{p'} =1$. 
\item We say that $w$ belongs to the Muckenhoupt class $\mathcal{A}_{1}$ if there exists a constant $C > 0$ such that
\[
Mw(x) \leq C w(x)
\]
for almost every $x \in \R^{n}$. 
\item The Muckenhoupt class $\mathcal{A}_{\infty}$ is defined as
\[
\mathcal{A}_{\infty} = \bigcup_{1<p<\infty} \mathcal{A}_{p}.
\]
\end{enumerate}
\end{defi}
During the rest of the paper we are going to focus mostly on the following well-known Muckenhoupt weights.

\blue{
  \begin{rmk}
    Since the pioneering work of Muckenhoupt \cite{Muck72,Muck72a,Muck73}, these classes of weight functions have been studied in great detail, we refer, in particular, to the monographs \cite{GC,Tor86,St93} for a complete account on the theory of Muckenhoupt weights. Among the various features of such weights we would like to mention a somehow surprising one, the so-called `reverse H\"older inequality': if $w\in \mathcal{A}_p$ with $p>1$, then there exists some number $r<p$ such that $w\in\mathcal{A}_r$ (the monotonicity in the other direction is clear). In our case this fact will re-emerge in the number 
\begin{equation}
r_0(w)  := \inf\{r\geq 1~: w\in\mathcal{A}_r\} , \quad w\in\mathcal{A}_\infty, 
\label{r0}
\end{equation}
that plays an essential r\^ole later on.    
\end{rmk}}

\begin{ex}\label{Ex-walpha}
  Let $ \alpha \in \R$ and define the weight $w_\alpha$ for $x=(x_{1},\dots,x_{n}) \in \R^{n}$ by
\begin{equation}\label{eq:weight w_alpha}
w_{\alpha}(x)=
\begin{cases}
\abs{x_n}^{\blue{\alpha}} &\text{ if } \abs{x_n}\leq 1, \\
1 &\text{ if } \abs{x_{n}} > 1.
\end{cases}
\end{equation}
For $ 1 < p < \infty$, the weight $w_{\alpha} \in \mathcal{A}_{p}$ if, and only if, $ -1 < \alpha < p-1$ and $w_{\alpha} \in \mathcal{A}_{1}$ if, and only if, $ -1 < \alpha \leq 0$. We write
\[
r_{0}(w_{\alpha}) = \inf\{\magenta{r \geq 1}: w_{\alpha} \in \mathcal{A}_{r}\} = \max\{\alpha+1,1\} = 
\begin{cases}
1 &\text{ if } -1 < \alpha \leq 0, \\
\alpha + 1 &\text{ if } \alpha > 0.
\end{cases}
\]
(See \cite[Proposition 2.8 and Remark 2.9/(b)]{HP}).
\end{ex}


Let $(\Omega, \mu)$ be a measure space, $ 0 < p < \infty$ and $ 0 < r \leq \infty$. We recall that the \emph{Lorentz space} $L_{p,r}(\Omega,\mu)$ is the set of all measurable functions $f: \Omega \longrightarrow \C$ with finite quasi-norm 
\[
\norm{f|L_{p,r}(\Omega,\mu)} = 
\begin{cases}
\Big(\int_{0}^{\infty} [t\mu_{f}(t)^{1/p}]^{r} \frac{\dint t}{t} \Big)^{1/r} &\text{ if } 0 < r < \infty, \\
\sup_{t>0} t\mu_{f}(t)^{1/p} &\text{ if } r=\infty,
\end{cases}
\]
where
\[
\mu_{f}(t) = \mu(\{\omega \in \Omega: \abs{f(\omega)} >t\}).
\]
For $p=r$, the Lorentz space $L_{p,p}(\Omega,\mu)$ coincides with the \emph{Lebesgue space} 
\[L_{p}(\Omega,\mu) = \Big\lbrace f: \Omega \rightarrow \C \text{ measurable }: \norm{f|L_{p}(\Omega,\mu)} = \Big(\int_{\Omega} \abs{f(\omega)}^{p} \dint\mu\Big)^{1/p} < \infty \Big\rbrace \]
with equivalent quasi-norms. (See, for example, \cite[Theorem~3.15]{T15a}, 
\cite[Proposition 1.4.9]{Gr}, or \cite[Chapter 2, Proposition 1.8]{BS}). If $w$ is a weight on $\R^{n}$ we simply write $L_{p,r}(\R^{n},w)$ for Lorentz spaces defined on the measure space $(\R^{n},w(x)\dint x)$ and $L_{p}(\R^{n},w)$ for the related Lebesgue spaces.

\blue{\begin{rmk}
Note that for $p=\infty$ one obtains the classical (unweighted) Lebesgue space, 
$L_\infty(\R^n, w) = L_\infty(\R^n)$, $w\in\mathcal{A}_\infty$, which explains that we restrict ourselves to $\ p<\infty$ in what follows.
\end{rmk}}

\begin{defi}
Let $(A, \norm{ \cdot  |A  })$ be a quasi-Banach space, $(\Omega,\mu)$ a measure space, $ 0 < p < \infty$ and $ 0 < r \leq \infty$. We define $L_{p,r}(A; \Omega,\mu)$ as the space of all (equivalence classes of) strongly measurable functions $f: \Omega \longrightarrow A$ which have finite quasi-norm
\[
\norm{f|L_{p,r}(A; \Omega,\mu)} = \norm{\norm{f(\cdot)|A} \ |L_{p,r}(\Omega,\mu)}.
\]
\end{defi}
See, for example, \cite{Kree, T78, BL}. 

\begin{ex}
  Now we collect 
  some concrete examples of these spaces that will 
  appear later.
\begin{itemize}
\item If $ A = \C$, then $L_{p,r}(A; \Omega,\mu) = L_{p,r}(\Omega,\mu)$.
\item If $p=r$, then $L_{p,p}(A; \Omega,\mu)$ is the Lebesgue space
\[
L_{p}(A; \Omega,\mu) = \Big\lbrace f: \Omega \longrightarrow A \text{ measurable }: \norm{f|L_{p}(A; \Omega,\mu)}=\norm{\norm{f(\cdot)|A} |L_{p}(\Omega,\mu)} < \infty \Big\rbrace
\]
with equivalent quasi-norms.
\item If $\Omega=\R^{n}$ and $\mu$ is the Lebesgue measure, we put $L_{p,r}(A) := L_{p,r}(A; \Omega,\mu)$.
\item If $w$ is a weight on $\R^{n}$, we put $L_{p,r}(A,w) := L_{p,r}(A; \R^{n}, w(x) \dint x)$. 
\item Let $\Omega = \N_{0}=\N \cup\{0\}$, $s \in \R$, $ A = \C$, and $\mu= \sum_{j \in \N_{0}} 2^{js}\delta_{\{j\}}$ the (weighted) counting measure, that is, where $\delta_{\{j\}}(B)=\begin{cases} 1, & j\in B,\\ 0, & j\not\in B, \end{cases}$\quad and $B\subseteq \Omega$ measurable. Then 
\[
\ell_{p}^{s}:= L_{p}(\C, \N_{0},\mu)=\Big\lbrace \lambda=(\lambda_{j})_{j\in\No}\subset \C: \norm{ \lambda |\ell_{p}^{s}} = \Big(\sum_{j \in \N_0} 2^{jsp} \abs{\lambda_{j}}^{p} \Big)^{1/p} < \infty \Big\rbrace.
\]

If $s=0$ we just write $\ell_{p}$. 
\item If $\Omega = \Z^{n}$, $\mu= \sum_{m \in \Z^{n}}\delta_{\{m\}}$ and $ A = \C$, then 
\[
\bl{p}:= L_{p}(\C, \Z^{n},\mu) :=  \Big\lbrace \lambda=(\lambda_{m})_{m\in\Z^n}\subset \C: \norm{\lambda |\bl{p}} = \Big(\sum_{m \in \Z^{n}}  \abs{\lambda_{m}}^{p} \Big)^{1/p} < \infty \Big\rbrace.
\]
\item If $\Omega = \N_{0}$, $s \in \R$, $\mu= \sum_{j \in \N_{0}} 2^{js}\delta_{\{j\}}$ and $ A = \bl{p}$, then we denote by $\ell_{q}^{s}(\bl{p})$ the space  
\begin{align*}
  \ell_{q}^{s}(\bl{p}):= & L_{q}(\bl{p}; \N_{0}, \mu) \\
  = & \Big\lbrace (\lambda_{j,m})\subset \C: \norm{(\lambda_{j})|\ell_{q}^{s}(\bl{p})} = \Big(\sum_{j \in \N_0} 2^{jsq} \Big(\sum_{m\in \Z^{n}}\abs{\lambda_{j,m}}^{p}\Big)^{q/p} \Big)^{1/q} < \infty \Big\rbrace,
\end{align*}
with the usual modification in case of $q=\infty$.
\end{itemize}
\end{ex}

We recall here two well-known assertions related to the boundedness of the Hardy-Littlewood maximal operator in \eqref{eq:maximal operator} that will be useful later. The first one corresponds to \cite[Formula(1.5)]{AJ}. 
\begin{lem}\label{lem:maximal inequality 1}
Let $ 1 < v < \infty$. Then there is a constant $C_{v} > 0$ such that
\[
\Big(\int_{\R^{n}} \abs{Mf(x)}^{v} w(x) \dint x \Big)^{1/v} \leq C_{v} \Big(\int_{\R^{n}} \abs{f(x)}^{v} w(x) \dint x \Big)^{1/v},
\]
for all $f\in L_{v}(\R^{n},w)$ if, and only if, $ w \in \mathcal{A}_{v}$. 
\end{lem}

The next one can be found in \cite[Theorem 3.1/b]{AJ} or \cite[Remark V.6.5]{GC}.

\begin{lem}\label{lem:maximal inequality 2}
Let $ 1 < r,v < \infty$. Then there is a constant $C_{r,v} > 0$ such that
\[
\Big( \int_{\R^{n}} \Big(\sum_{j\in \N_0} \abs{Mf_{j}(x)}^{r} \Big)^{v/r} w(x) \dint x \Big)^{1/v} \leq C_{r,v} \Big( \int_{\R^{n}} \Big(\sum_{j\in \N_0} \abs{f_{j}(x)}^{r} \Big)^{v/r} w(x) \dint x \Big)^{1/v},
\]
for all $(f_{j}) \in L_{v}(\ell_{r},w)$ if, and only if, $w \in \mathcal{A}_{v}$. 
\end{lem}

Now we deal with the function spaces we have in mind, i.e., Besov and Triebel-Lizorkin spaces. 
Let $\Sc(\R^{n})$ be the Schwartz space of all complex-valued rapidly decreasing infinitely differentiable functions on $\R^{n}$. By $\Sc'(\R^{n})$ we denote the space of all tempered distributions on $\R^{n}$ and for any $f \in \Sc'(\R^{n})$ we put $\hat{f}$ for its Fourier transform and $f^{\vee}$ for its inverse Fourier transform. 

Let $\varphi_{0} \in \Sc(\R^{n})$ with
\[
\varphi_{0}(x) = 1 \text{ if } \abs{x}\leq 1 \quad \text{and} \quad \varphi_{0}(x) = 1 \text{ if } \abs{x}\geq 3/2,
\] 
and for $ j \in \N$ put $\varphi_{j}(x) = \varphi_{0}(2^{-j}x) - \varphi_{0}(2^{-j+1}x)$, $ x \in \R^{n}$. Since
\[
\sum_{j=0}^{\infty} \varphi_{j}(x) = 1 \quad \text{for all } x \in \R^{n},
\]
the sequence $(\varphi_{j})_{j \in \N_{0}}$ forms a dyadic resolution of unity.
\begin{defi}\label{def:Besov spaces}
Let $ s \in \R$, $ 0 < p,q\leq \infty$. The \emph{Besov space} $B_{p,q}^{s}(\R^{n})$ is formed by all $ f \in \Sc'(\R^{n})$ having a finite quasi-norm
\[
\norm{f|B_{p,q}^{s}(\R^{n})} = \Big(\sum_{j=0}^{\infty} 2^{jsq} \norm{(\varphi_{j} \hat{f})^{\vee}|L_{p}(\R^{n})}^{q} \Big)^{1/q} < \infty,
\]
with the usual modifications if $p=\infty$ and/or $q=\infty$. 
\end{defi}

\noindent{\em Convention.}\quad If $p=q$, sometimes we simply write $B_{p}^{s}(\R^{n})$ instead of $B_{p,p}^{s}(\R^{n})$. 

  \blue{
\begin{rmk}\label{rem-Besov}
The  spaces $B_{p,q}^s(\mathbb{R}^n)$ are
independent of the particular choice of the smooth dyadic resolution of unity
appearing in their definition. They are quasi-Banach spaces
(Banach spaces for $p,q\geq 1$), and $ \mathcal{S}(\R^n) \hookrightarrow B^s_{p,q}(\R^n)\hookrightarrow \mathcal{S}'(\R^n)$, where the first embedding is dense if $\magenta{0< p,q<\infty}$; we refer, in particular, to the series of monographs \cite{T83,T92,T06} for a comprehensive treatment of the spaces. 
There is a parallel approach when interchanging in the above norm the $L_p$ and $\ell_q$ norm, this leads to the scale of Triebel-Lizorkin spaces $F^s_{p,q}(\R^n)$. We postpone their formal definition to the next section when we shall deal with their weighted counterparts.
\end{rmk}
}

Now we review some results related to the wavelet representation of Besov spaces $B_{p,q}^{s}(\R^{n})$. We follow mainly the notation in \cite[Section 1.2.2]{T08} and \cite{HST}.

For $ u \in \N$, let $C^{u}(\R^{n})$ be the space of all complex-valued continuous functions on $\R$ having continuous bounded derivatives up to order $u$ inclusively. Let $\psi_{F} \in C^{u}(\R)$ and $\psi_{M} \in C^{u}(\R)$ be real compactly supported Daubechies wavelets with 
\[
\int_{\R} \psi_{M}(x)x^{v} \dint x=0 \quad \text{ for all } v \in \N_{0}, \ v < u.
\]
We enumerate the set $\{F,M\}^{n} =\{G_{1},\dots,G_{2^n}\}$ with $G_{1}=(F,F,\dots,F)$ and $G_{\ell} = (G_{\ell}^{1},\dots,G_{\ell}^{n})$ where $G_{\ell}^{r} \in \{F,M\}$ for all $r=1,\dots,n$ and $\ell = 2,3,\dots,2^{n}$. Put
\begin{equation}\label{eq:wavelets}
\psi_{m}(x) = \prod_{r=1}^{n} \psi_{F}(x_{r}-m_{r}) \quad \text{and} \quad \psi_{\ell,m}^{j}(x) = 2^{jn/2} \prod_{r=1}^{n} \psi_{G_{\ell}^{r}}(2^{j}x_{r}-m_{r}),
\end{equation}
for $j \in \N_{0}$, $m = (m_{1},\dots,m_{n}) \in \Z^{n}$, $\ell = 2,3,\dots,2^{n}$ and $ x \in \R^{n}$.

For any $ j \in \N_0$ and $m \in \Z^{n}$ we define
\[
Q_{j,m} = Q_{j,m}^{(n)} = 2^{-j}m+2^{-j-1}(-1,1)^{n} = \prod_{r=1}^{n} (2^{-j}m_{r}-2^{-j-1}, 2^{-j}m_{r}+2^{-j-1}),
\]
that is to say, dyadic cubes centered at $2^{-j}m$ of side length $2^{-j}$. We denote by $\chi_{j,m}$ the characteristic function of $Q_{j,m}$.
\begin{rmk}\label{rmk:wavelet support}
Observe that as $\psi_{F}$ and $\psi_{M}$ are compactly supported, there exists $C> 0$ such that $\supp \psi_{F} \subset (-C,C)$ and $\supp \psi_{M} \subset (-C, C)$. Therefore
\[
\supp \psi_{m} \subset 2CQ_{0,m} \quad \text{and} \quad \supp \psi_{\ell,m}^{j} \subset 2C Q_{j,m}^{\ell},
\]
for every $j\in \N_{0}$, $m \in \Z^{n}$ and $\ell = 2,3,\dots,2^{n}$.
\end{rmk}
Now we introduce the sequence space related to the wavelet representation of $B_{p,q}^{s}(\R^{n})$. 
\begin{defi}\label{def:Besov sequence space}
Let $ s \in \R$ and $ 0 < p,q \leq \infty$. We define $b_{p,q}^{s}(\R^{n})$ as the collection of all sequences 
\[\lambda = \Big((\lambda_{m})_{m \in \Z^{n}}, (\lambda_{m}^{j,2})_{\substack{j \in \N_{0} \\ m \in \Z^{n}}}, \dots, (\lambda_{m}^{j,2^{n}})_{\substack{j \in \N_{0} \\ m \in \Z^{n}}} \Big), \]
quasi-normed by 
\[
\norm{\lambda|b_{p,q}^{s}(\R^{n})} = \Big(\sum_{m\in \Z^{n}} \abs{\lambda_{m}}^{p}\Big)^{1/p} + \sum_{\ell=2}^{2^{n}} \Big( \sum_{j \in \N_0} 2^{j(s-n/p)q} \Big(\sum_{m \in \Z^{n}}  |\lambda_{m}^{j,\ell}|^{p} \Big)^{q/p} \Big)^{1/q}
\]
with the usual modifications if $p=\infty$ and/or $q=\infty$.
\end{defi}

For $\alpha \in \R$ we put $\alpha_{+} = \max\{\alpha,0\}$. The following characterization of Besov spaces $B_{p,q}^{s}(\R^{n})$ in terms of wavelets can be found in \cite[Theorem 1.20]{T08}.
\begin{thm}\label{thm:wavelet Besov spaces}
Let $ s \in \R$, $0 < p,q \leq \infty$ and let $\{\psi_{m}, \psi_{\ell,m}^{j}: m \in \Z^{n}, \ j \in \N_0, \ \ell = 2,3,\dots,2^{n}\}$ be the wavelets in \eqref{eq:wavelets} with $u > \max\{s, n\Big(\frac{1}{p}-1\Big)_{+}-s\}$. Let $f \in \Sc'(\R^{n})$. Then $f \in B_{p,q}^{s}(\R^{n})$ if, and only if, it can be represented as
\[
f = \sum_{m\in \Z^{n}} \lambda_{m} \psi_{m} + \sum_{\ell=2}^{2^{n}} \sum_{j\in \N_{0}}\sum_{m \in \Z^{n}} \lambda_{m}^{j,\ell} 2^{-jn/2} \psi_{\ell,m}^{j},\quad \text{with}\quad \lambda \in b_{p,q}^{s}(\R^{n}),
\]
unconditional convergence being in $\Sc'(\R^{n})$. The representation is unique with $\lambda_{m} = (f,\psi_{m})$ and $\lambda_{m}^{j,\ell} = 2^{jn/2} (f, \psi_{\ell,m}^{j})$ and the operator 
\[I: f \rightarrow \Big((\lambda_{m})_{m \in \Z^{n}}, (\lambda_{m}^{j,2})_{\substack{j \in \N_{0} \\ m \in \Z^{n}}}, \dots, (\lambda_{m}^{j,2^{n}})_{\substack{j \in \N_{0} \\ m \in \Z^{n}}} \Big) \]
is an isomorphism from $B_{p,q}^{s}(\R^{n})$ onto $b_{p,q}^{s}(\R^{n})$. 
\end{thm}


\section{Triebel-Lizorkin spaces with weight $w_{\alpha}$}\label{F-weight}
Weighted Triebel-Lizorkin spaces have been systematically studied in \cite{Bui82,Bui84} with subsequent papers \cite{Bui96, Bui97}. We focus here on Triebel-Lizorkin spaces $F_{p,q}^{s}(\R^{n},w_{\alpha})$ where $w_{\alpha}$ is the weight defined in \eqref{eq:weight w_alpha}. The spaces $F_{p,q}^{s}(\R^{n},w_{\alpha})$ are also a particular case of the spaces studied in \cite{HP,P}.
\begin{defi}
Let $0 < p < \infty$, $0 < q \leq \infty$, $\alpha > -1$, $s \in \R$ and $(\varphi_{j})$ a smooth dyadic resolution of unity. \emph{The weighted Triebel–Lizorkin space} $F_{p,q}(\R^{n},w_{\alpha})$ is the set of all distributions $f \in \Sc'(\R^{n})$ such that
\[
\norm{f|F_{p,q}^{s}(\R^{n},w_{\alpha})} = \Big\|\Big(\sum_{j=0}^{\infty} 2^{jsq} \abs{(\varphi_{j} \hat{f})^{\vee}(\cdot)}^{q} \Big)^{1/q}|L_{p}(\R^{n}, w_{\alpha})\Big\|< \infty
\]
with the usual modification if $q = \infty$.
\end{defi}

\begin{rmk}\label{rem-Triebel}
\blue{The  spaces $F_{p,q}^s(\mathbb{R}^n,w)$ (as well as their Besov space counterparts) are
independent of the particular choice of the chosen smooth dyadic resolution of unity, they are quasi-Banach spaces (Banach spaces for $p,q\geq 1$), and the embedding of $\mathcal{S}(\R^n)$ is dense in $F^s_{p,q}(\R^n,
w)$ for $q<\infty$. In case of $w\in\mathcal{A}_\infty$ these spaces have been studied first by Bui in \cite{Bui82,Bui84}, with
subsequent papers \cite{Bui96,Bui97}. It turned out that many of the
results from the unweighted situation have weighted counterparts: e.g., we
have $F^0_{p,2}(\R^n,w) = h_p(\R^n,w)$, $0<p<\infty$, where  the latter are
Hardy spaces, see \cite{Bui82}, and, in particular, $h_p(\R^n, w) =
L_p(\R^n, w) = F^0_{p,2}(\R^n,w)$, $ 1<p<\infty$, $\ w\in \mathcal{A}_p$. Concerning (classical) Sobolev spaces $W^k_p(\R^n, w) $ (built
upon $L_p(\R^n, w)$ in the usual way) it holds  $W^k_p(\R^n, w) = F^k_{p,2}(\R^n,
w)$, $k\in\N_0$, $1<p<\infty$, $w\in \mathcal{A}_p$, cf. \cite{Bui82}. 
}

  Observe that if $\alpha=0$, the spaces $F_{p,q}^{s}(\R^{n}, w_{0})$ coincide with the classical Triebel-Lizorkin spaces $F_{p,q}^{s}(\R^{n})$, briefly mentioned in Remark~\ref{rem-Besov} already. 
\end{rmk}

\subsection{Atomic decomposition of spaces $F_{p,q}^{s}(\R^{n},w_{\alpha})$}
Now we recall the atomic decomposition of spaces $F_{p,q}^{s}(\R^{n},w_{\alpha})$ given in \cite{HP}.
\begin{defi}\label{def:atom}
\begin{enumerate}[\bfseries (a)]
\item Suppose that $K \in \N_{0}$ and $b > 1$. The complex-valued function $ a \in C^{K}(\R^{n})$ is said to be an \emph{$1_{K}$-atom} if the following assumptions are satisfied: 
\begin{enumerate}[(i)]
\item $\supp a \subset b Q_{0,m}$ for some $ m \in \Z^{n}$, 
\item $\abs{\Dd^{\beta} a(x)} \leq 1$ for $\abs{\beta} \leq K$, $ x \in \R^{n}$. 
\end{enumerate}
\item Suppose that $s \in \R$, $ 0 < p \leq \infty$, $K, L\in \N_0$ and $b > 1$. The complex valued function $ a \in C^{K}(\R^{n})$ is said to be an \emph{$(s,p)_{K,L}$-atom} if for some $j \in \N_0$ the following assumptions are satisfied
\begin{enumerate}[(i)]
\item $\supp a \subset b Q_{j,m}$ for some $m \in \Z^{n}$, 
\item $\abs{\Dd^{\beta}a(x)} \leq 2^{-j(s-n/p) + \abs{\beta}j}$ for $\abs{\beta} \leq K$ and $ x \in \R^{n}$,
\item $\int_{\R^{n}} x^{\gamma} a(x) \dint x =0$ for $\abs{\gamma} < L$. 
\end{enumerate}
\end{enumerate}
\blue{In the sequel we write $a_{j,m}$ instead of $a$ if the atom is located at  $Q_{j,m}$, i.e.,  $\supp a_{j,m}\subset bQ_{j,m}$, $j\in\N_0$, $m\in\Z^n$.}
\end{defi}

Our aim is some decomposition of elements from $F^s_{p,q}(\R^n, w_\alpha)$ by atoms, similar to the wavelet decomposition recalled in Theorem~\ref{thm:wavelet Besov spaces}. For that reason we also need 

\begin{defi}\label{defi:weighted T-L sequences for atoms}
Let $ 0 < p < \infty$, $ 0 < q \leq \infty$ and $ \alpha > -1$. We define $f_{p,q}(\R^{n},w_{\alpha})$ as the set of all sequences $(\lambda_{j,m})$ of complex numbers with finite quasi-norm
\[
\norm{\lambda|f_{p,q}(\R^n, w_{\alpha})} = \Big\|\Big(\sum_{j \in \N_{0}} \sum_{m \in \Z^{n}} \abs{\lambda_{j,m} \chi_{j,m}^{(p)} (\cdot)}^{q} \Big)^{1/q}|{L_{p}(\R^{n}, w_{\alpha})}\Big\| 
\]
where $\chi_{j,m}^{(p)}(x) = 2^{jn/p} \chi_{j,m}(x)$. 
\end{defi}
Subsequently, given an arbitrary index set $I$ and two sets of positive numbers $\{a_{i}: i \in I\}$ and $\{b_{i}: i \in I\}$, we write $a_{i} \lesssim b_{i}$ if there is a positive constant $c$ such that $a_{i}\leq c b_{i}$ for all $i \in I$. We put $a_{i} \sim b_{i}$ if $a_{i}\lesssim b_{i}$ and $b_{i}\lesssim a_{i}$. \blue{Recall our notation $r_0(w_\alpha)$ as introduced in \eqref{r0}, see also Example~\ref{Ex-walpha}.}

\begin{thm}\label{thm:atomic decomp}
Let $0 < p < \infty$, $ 0 < q \leq \infty$, $ s \in \R$ and $ \alpha >-1$. Let $K, L\in \N_0$ with $K > s$ and $L > n \Big( \frac{1}{\min(\frac{p}{r_0(w_{\alpha})},q)}-1 \Big)_{+}$. A tempered distribution $f \in \Sc'(\R^{n})$ belongs to $F_{p,q}^{s}(\R^{n}, w_{\alpha})$ if, and only if, it can be written as a series 
\begin{align}\label{eq:atomic decomp}
f (x) = \sum_{j\in\N_0} \sum_{m \in \Z^{n}} \lambda_{j,m} a_{j,m}(x) \text{ converging in } \Sc'(\R^{n}), 
\end{align}
where $a_{j,m}(x)$ are $1_{K}$-atoms ($j=0$) or $(s,p)_{K,L}$-atoms ($j \in \N$) and $\lambda \in f_{p,q}(\R^n, w_{\alpha})$. Furthermore
\[
\norm{f|F_{p,q}^{s}(\R^{n},w_{\alpha})} \sim \inf \{ \norm{\lambda|f_{p,q}(\R^{n}, w_{\alpha})}\},
\]
where the infimum is taken over all admissible representations \eqref{eq:atomic decomp}. 
\end{thm}

\blue{We refer to \cite[Theorem~3.10]{HP} for a proof, see also \cite{BoHo}.}

\begin{prop}\label{prop:equivalent f norm}
Let $ 0 < p < \infty$, $ 0 < q \leq \infty$, $\alpha > -1$ and let $(E_{j,m})$ be a sequence of Lebesgue measurable sets on $\R^{n}$ each of them included in the corresponding dyadic cube $Q_{j,m}$ and satisfying that $\abs{E_{j,m}} \sim \abs{Q_{j,m}}$ for every $j \in \N_0$ and $m \in \Z^{n}$. Then
\[
\norm{\lambda | f_{p,q}(\R^{n},w_{\alpha})} \sim \Big\|\Big(\sum_{j\in\N_0} \sum_{m \in \Z^{n}} \abs{\lambda_{j,m}}^{q} 2^{jnq/p} \chi_{E_{j,m}}(\cdot) \Big)^{1/q} |L_{p}(\R^{n},w_{\alpha})\Big\|,
\]
where $\chi_{E_{j,m}}$ stands for the characteristic function of the set $E_{j,m}$.
\end{prop}

\begin{proof}
We follow the ideas in the proof of \cite[Proposition 2.7]{FJ}. As $E_{j,m} \subset Q_{j,m}$ for every $j \in \N_{0}$ and $m \in \Z^{n}$, it is straightforward that
\[
\Big\|\Big( \sum_{j\in \N_0} \sum_{m \in \Z^{n}} \abs{\lambda_{j,m}}^{q} 2^{jnq/p} \chi_{E_{j,m}}(\cdot) \Big)^{1/q}|L_{p}(\R^{n},w_{\alpha}) \Big\| \leq \norm{\lambda|f_{p,q}(\R^{n},w_{\alpha})}.
\]
We prove now the reverse inequality. Note that for every $j \in \N_0$ and $m \in \Z^{n}$, $\chi_{j,m} \lesssim M (\chi_{E_{j,m}})$.  We assume first that $ 0 < q < \infty$. Taking $ 0 < A < \min\Big(\frac{p}{r_{0}(w_{\alpha})}, q\Big)$ and applying Lemma \ref{lem:maximal inequality 2}, we obtain that
\begin{align*}
&\Big\| \Big( \sum_{j\in\N_0} \sum_{m \in \Z^{n}} \abs{\lambda_{j,m}}^{q} 2^{jnq/p} \chi_{{j,m}}(\cdot) \Big)^{1/q} |L_{p}(\R^{n},w_{\alpha})\Big\| \\
=&  \Big\| \Big( \sum_{j\in\N_0} \sum_{m \in \Z^{n}} (\abs{\lambda_{j,m}}^{A} 2^{jnA/p})^{q/A} \chi_{j,m}(\cdot) \Big)^{A/q}|L_{p/A}(\R^{n},w_{\alpha}) \Big\|^{1/A}\\
\lesssim & \Big\| \Big( \sum_{j\in\N_0} \sum_{m \in \Z^{n}} (\abs{\lambda_{j,m}}^{A} 2^{jnA/p})^{q/A} M\chi_{E_{j,m}}(\cdot)\Big)^{A/q}|L_{p/A}(\R^{n}, w_{\alpha}) \Big\|^{1/A}\\
= & \Big\| \Big( \sum_{j\in\N_0} \sum_{m \in \Z^{n}} M\Big((\abs{\lambda_{j,m}}^{A} 2^{jnA/p})^{q/A} \chi_{E_{j,m}}\Big)(\cdot) \Big)^{A/q}| L_{p/A}(\R^{n}, w_{\alpha})\Big\|^{1/A} \\
\lesssim & \Big\| \Big( \sum_{j \in \N_{0}} \sum_{m \in \Z^{n}} (\abs{\lambda_{j,m}}^{A} 2^{jnA/p})^{q/A} \chi_{E_{j,m}}(\cdot)\Big)^{A/q}|L_{p/A}(\R^{n}, w_{\alpha}) \Big\|^{1/A}\\
= & \Big\| \Big( \sum_{j\in \N_0} \sum_{m \in \Z^{n}} \abs{\lambda_{j,m}}^{q} 2^{jnq/p} \chi_{E_{j,m}}(\cdot) \Big)^{1/q} |L_{p}(\R^{n}, w_{\alpha}) \Big\|.
\end{align*}
Now we study the case $q=\infty$. For $ 0 < A < \frac{p}{r_{0}(w_{\alpha})}$, 
\begin{align*}
\norm{\lambda|f_{p,\infty}(\R^{n},w_{\alpha})} &= \Big\|{\sup_{j,m} |\lambda_{j,m}| 2^{jn/p} \chi_{{j,m}}(\cdot)}| L_{p}(\R^{n}, w_{\alpha})\Big\| \\
&= \Big\|\sup_{j,m} |\lambda_{j,m}|^{A} 2^{jnA/p} \chi_{{j,m}}(\cdot)|L_{p/A}(\R^{n}, w_{\alpha}) \Big\|^{1/A} \\
&\lesssim \Big\|\sup_{j,m} |\lambda_{j,m}|^{A} 2^{jnA/p} M\chi_{E_{j,m}}(\cdot)|L_{p/A}(\R^{n}, w_{\alpha}) \Big\|^{1/A} \\
&\leq \Big\|M(\sup_{j,m} |\lambda_{j,m}|^{A} 2^{jnA/p} \chi_{E_{j,m}}) (\cdot)|L_{p/A}(\R^{n}, w_{\alpha}) \Big\|^{1/A} \\
& \lesssim \Big\|\sup_{j,m} |\lambda_{j,m}|^{A} 2^{jnA/p} \chi_{E_{j,m}}(\cdot)|L_{p/A}(\R^{n}, w_{\alpha})\Big\|^{1/A} \\
&=  \Big\|\sup_{j,m} |\lambda_{j,m}| 2^{jn/p} \chi_{E_{j,m}} (\cdot)|L_{p}(\R^{n}, w_{\alpha}) \Big\|,
\end{align*}
where we have used Lemma \ref{lem:maximal inequality 1}. 
\end{proof}
\subsection{Trace and extension operators on $F_{p,q}^{s}(\R^{n},w_{\alpha})$}
Let $n \in \N$ and $k=1,2,\dots,n-1$. If $x=(x_{1},\dots,x_{n}) \in \R^{n}$ and $\varphi \in \Sc(\R^{n})$, the trace operator $\trk$ is defined as
\[
\trk: \varphi(x) \longrightarrow \varphi(x_1,\dots,x_{n-k},0,0,\dots,0).
\]

Let $A$ and $B$ be two-quasi Banach spaces such that
\begin{align*}
\Sc(\R^{n}) \emb A \emb \Sc'(\R^{n}) \quad \text{and} \quad \Sc(\R^{n-k}) \emb B \emb \Sc'(\R^{n-k}).
\end{align*}
Suppose that there exists $c>0$ satisfying
\begin{equation}\label{eq:trace definition}
\norm{\trk \varphi|B} \leq c \norm{\varphi|A} \quad \text{for every } \varphi \in \Sc(\R^{n}).
\end{equation}
Although the punctual definition of the trace operator might not make sense in general on $A$, if $\Sc(\R^{n})$ is dense, the inequality \eqref{eq:trace definition} can be uniquely extended by completion to the whole space $A$. Due to the density of $\Sc(\R^{n})$ in $F_{p,q}^{s}(\R^{n},w_{\alpha})$ if $0 < p,q < \infty$ (see \cite{Bui82}), this is how we are going to understand the operator $\trk$ acting on spaces $F_{p,q}^{s}(\R^{n},w_{\alpha})$. \blue{In case of $q=\infty$ one may strengthen an embedding argument as explained in some detail in \cite[Rem.~1.170]{T06}\magenta{, but using now that if $ \eps > 0$, then $F_{p,\infty}^{s}(\R^{n}, w_\alpha) \hookrightarrow F_{p,p}^{s-\eps}(\R^{n}, w_{\alpha})$ (see \cite[Theorem 2.6/(i)]{Bui82}).}}


\begin{thm}\label{thm:trace 1}
Let $0 < p < \infty$, $ 0 < q \leq \infty$, $ \alpha > -1$ and $s-\frac{\alpha+1}{p}>(n-1)(\frac{1}{p}-1)_{+}$. Then 
\[
\tr F_{p,q}^{s}(\R^{n}, w_{\alpha}) = \tr B_{p}^{s-\frac{\alpha}{p}}(\R^{n}) = B_{p}^{s-\frac{\alpha+1}{p}}(\R^{n-1}).
\]
\end{thm}
\begin{proof}
This result corresponds to \cite[Proposition 2.4]{HM} with a correction in the values that $s$ can take, regarding \cite[Theorem 4.4]{P} with $d=n-1$ and $\Gamma = \R^{n-1}$ and \cite[Section 4.4.1/(i)]{T92}.
\end{proof}
Here $\tr F_{p,q}^{s}(\R^{n}, w_{\alpha}) =  B_{p}^{s-\frac{\alpha+1}{p}}(\R^{n-1})$ means that, \blue{firstly, the trace operator acts linear and bounded from  $F_{p,q}^{s}(\R^{n}, w_{\alpha})$ into $B_{p}^{s-\frac{\alpha+1}{p}}(\R^{n-1})$, and, secondly, that} for any $g \in B_{p}^{s-\frac{\alpha+1}{p}}(\R^{n-1})$ there exists an $f \in F_{p,q}^{s}(\R^{n},w_{\alpha})$ such that $\tr f=g$ and 
\[
\norm{g|B_{p}^{s-\frac{\alpha+1}{p}}(\R^{n-1})} \sim \inf \{ \norm{f|F_{p,q}^{s}(\R^{n},w_{\alpha})}: \tr f = g\}.
\]
\begin{thm}\label{thm:trace 2}
Let $0 < p < \infty$, $ 0 < q \leq \infty$, $k\in \{1, \dots, n-1\}$, $\alpha > -1$ and $ s - \frac{\alpha + k}{p} > (n-k)(\frac{1}{p}-1)_{+}$. Then 
\begin{enumerate}[\upshape\bfseries (i)]
\item $\trk: F_{p,q}^{s}(\R^{n}, w_{\alpha}) \hookrightarrow B_{p}^{s-\frac{\alpha + k}{p}}(\R^{n-k})$ is \blue{linear and} bounded, and 
\item there exists a linear and bounded extension operator
\[
\ext: B_{p}^{s-\frac{\alpha + k}{p}}(\R^{n-k}) \rightarrow F_{p,q}^{s}(\R^{n}, w_{\alpha})
\]
such that $\ \trk \circ \ext = \id: B_{p}^{s-\frac{\alpha + k}{p}}(\R^{n-k}) \rightarrow B_{p}^{s-\frac{\alpha + k}{p}}(\R^{n-k})$. 
\end{enumerate}
In particular, 
\[\trk F_{p,q}^{s} (\R^{n}, w_{\alpha}) = B_{p}^{s-\frac{\alpha + k}{p}}(\R^{n-k}).\]
\end{thm}
\begin{proof}
As $\trk = \text{tr}_{\R^{(n-(k-1))-1}} \circ \dots \circ \text{tr}_{\R^{(n-1)-1}} \circ \text{tr}_{\R^{n-1}}$, applying Theorem \ref{thm:trace 1} and \cite[Section 4.4.1]{T92} we have (i). 

For proving (ii), we are going to build the extension operator using the wavelet representation of Besov spaces given in Theorem \ref{thm:wavelet Besov spaces} and the atomic decomposition of weighted Triebel-Lizorkin spaces in Theorem \ref{thm:atomic decomp}. We follow some of the ideas in \cite[Section 2.2.2]{T20}, based on \cite[Section~5.1.3]{T08}. 

Take $f \in B_{p}^{s-\frac{\alpha + k}{p}}(\R^{n-k})$ and consider its wavelet representation of order 
\[ u > \max \Big\lbrace s, (n-k)\Big(\frac{1}{p}-1\Big)_{+}-s+\frac{\alpha+k}{p}, n\Big(\frac{1}{\min(p/r_{0}(w_\alpha),q)}-1\Big)_{+}-s\Big\rbrace, \] 
\begin{align*}
f(x) = \sum_{m \in \Z^{n-k}} \lambda_{m} \psi_{m}(x) + \sum_{\ell=2}^{2^{n-k}} \sum_{j\in \N_{0}} \sum_{m \in \Z^{n-k}} \lambda_{m}^{j,\ell} 2^{-j(n-k)/2} \psi_{\ell,m}^{j}(x), \quad \text{for every } x \in \R^{n-k},
\end{align*}
where 
\begin{align*}
&\psi_{m}(x) = \prod_{r=1}^{n-k} \psi_{F}(x_{r}-m_{r}) &\text{ and }&  &\psi_{\ell,m}^{j}(x) = 2^{j(n-k)/2} \prod_{r=1}^{n-k} \psi_{G_{\ell}^{r}} (2^{j} x_{r}-m_{r}), \\
&\lambda_{m} = \lambda_{m}(f) = (f, \psi_{m}) &\text{ and }&  &\lambda_{m}^{j,\ell} = \lambda_{m}^{j,\ell}(f) = 2^{j(n-k)/2}(f,\psi_{\ell,m}^{j}),
\end{align*}
\blue{and
\begin{equation}\label{ddh-1}
\left\| f | B_{p}^{s-\frac{\alpha + k}{p}}(\R^{n-k})\right\| \sim \left\|\lambda| b^{s-\frac{\alpha + k}{p}}_{p,q}(\R^{n-k})\right\|
\end{equation}
in view of Theorem~\ref{thm:wavelet Besov spaces}.}

Take $\chi \in C_{0}^{\infty}(\R^{k})$ with $\supp \chi \subset (-1,1)^{k}$ and $\chi(y) = 1$ if $y \in [-1/2, 1/2]^{k}$. For every $x \in \R^{n-k}$ and $y \in \R^{k}$ we define
\begin{align*}
\ext f (x,y) &= \sum_{m \in \Z^{n-k}} \lambda_{m} \psi_{m}(x) \chi(y) + \sum_{\ell=2}^{2^{n-k}} \sum_{j\in \N_0} \sum_{m \in \Z^{n-k}} \lambda_{m}^{j, \ell} 2^{-j(n-k)/2} \psi_{\ell,m}^{j}(x) \chi(2^{j}y).
\end{align*}
We can rewrite it as $\ext f(x,y) = \sum_{\ell=1}^{2^{n-k}} g_{\ell}(x,y)$ with 
\[g_{\ell}(x,y) = \sum_{j\in \N_0} \ \sum_{(m,M) \in \Z^{n-k}\times\Z^k} \tilde{\lambda}_{j,(m,M)}^{\ell} a_{j,(m,M)}^{\ell}(x,y), \quad \blue{\ell=1, \dots, 2^{n-k}},\]
and being
\begin{align*}
a_{j,(m,M)}^{\ell}(x,y) &= \begin{cases}
\psi_{m}(x) \chi(y) &\text{ if } \ell =1, j = 0 \text{ and } M =0, \\
2^{-j(s-n/p)} 2^{-j(n-k)/2} \psi_{\ell,m}^{j}(x) \chi(2^{j}y) &\text{ if } \ell = 2,\dots,2^{n-k} \text{ and } M =0, \\
0 &\text{ otherwise,}
\end{cases}\\
\intertext{and}\\
\tilde{\lambda}_{j,(m,M)}^{\ell} &= \begin{cases} \lambda_{m} &\text{ if } \ell=1, j=0 \text{ and } M=0, \\
2^{j(s-n/p)}\lambda_{m}^{j,\ell} &\text{ if } \ell=2,\dots,2^{n-k} \text{ and } M=0, \\
0 &\text{ otherwise.} \end{cases}
\end{align*}
Now we prove that these are atomic decompositions of $g_{\ell}$ in $F_{p,q}^{s}(\R^{n}, w_{\alpha})$ according to Theorem \ref{thm:atomic decomp}. 
\begin{enumerate}
\item Let $\ell=1$. Then 
\begin{enumerate}[(i)]
\item $\supp a_{0,(m,0)}^{1} \subset \supp \psi_{m} \times (-1,1)^{k} \subset b Q_{0,(m,0)}$ \quad using Remark \ref{rmk:wavelet support}, and  
\item $\abs{\Dd^{\beta}a_{0,(m,0)}^{1}(x,y)} =\abs{\prod_{r=1}^{n-k} \Dd^{\beta_r} \psi_{F}(x_{r}-m_{r}) \Dd^{(\beta_{n-k+1},\dots,\beta_{n})} \chi(y)} \leq C$ \quad for all $(x,y) \in \R^{n}$ and $\abs{\beta} \leq u$. 
\end{enumerate}
\item Let $\ell=2,\dots,2^{n-k}$. Then 
\begin{enumerate}[(i)]
\item $\supp a_{j,(m,0)}^{\ell} \subset \supp \psi_{\ell,m}^{j} \times (-2^{-n}, 2^{-n})^{k} \subset b Q_{j,(m,0)}$ \quad using Remark \ref{rmk:wavelet support}, and 
\item 
$\displaystyle\begin{aligned}[t]
\abs{\Dd^{\beta} a_{j,(m,0)}^{\ell}(x,y)} &= 2^{-j(s-\frac{n}{p})+j\abs{\beta}} \Big|\prod_{r=1}^{n-k} \Dd^{\beta_{r}} \psi_{G_{\ell}^{r}}(2^{j}x_{r}-m_{r})\Dd^{(\beta_{n-k+1},\dots,\beta_{n})} \chi(2^{j}y)\Big|\\
&\leq C 2^{-j(s-\frac{n}{p})+j \abs{\beta}}, \quad \text{for all } (x,y) \in \R^{n} \text{ and } \abs{\beta} \leq u,
\end{aligned}$.
\item $\int_{\R^{n}} z^{\gamma} a_{j,(m,0)}^{\ell}(z) \dint z =0$, for every $ \abs{\gamma} < u$, \quad in view of the vanishing moments for the wavelets.
\end{enumerate}
\end{enumerate}
Thus we get that
\begin{align}\label{eq:estimate g1}
\norm{g_{1}|F_{p,q}^{s}(\R^{n}, w_{\alpha})} &\lesssim \|\tilde{\lambda}^{1}| f_{p,q}(\R^{n},w_{\alpha})\|= \Big\|{\Big( \sum_{m \in \Z^{n-k}} \abs{\lambda_{m}}^{q} \chi_{Q_{0,(m,0)}}(\cdot)\Big)^{1/q}}|L_{p}(\R^{n}, w_{\alpha})\Big\| \nonumber\\
&\sim \Big(\sum_{m \in \Z^{n-k}} \abs{\lambda_{m}}^{p} \int_{Q_{0,(m,0)}} w_{\alpha}(z) \dint z \Big)^{1/p} \sim \Big(\sum_{m \in \Z^{n-k}} \abs{\lambda_{m}}^{p} \Big)^{1/p}, 
\end{align}
since $\alpha>-1$.

Let $E_{j,m} = Q_{j,m}^{(n-k)} \times (2^{-j-2}, 2^{-j-1})^{k} \subset Q_{j,(m,0)}^{(n)}$. For $\ell=2,\dots, 2^{n-k}$, using Proposition \ref{prop:equivalent f norm}, the fact that $E_{j,m}$ are disjoint and $\int_{E_{j,m}} w_{\alpha}(z) \dint z \sim 2^{-j(\alpha+n)}$, we have
\begin{align}\label{eq:estimate gl}
\norm{g_{\ell}|F_{p,q}^{s}(\R^{n}, w_{\alpha})} &\lesssim \|{\tilde{\lambda}^{\ell}}|f_{p,q}(\R^{n},w_{\alpha})\| \nonumber\\
&= \Big\|{\Big(\sum_{m \in \Z^{n-k}} \sum_{j \in \N_{0}} |{\lambda_{m}^{j, \ell}}|^{q} 2^{j(s-n/p)q} 2^{jnq/p} \chi_{Q_{j,(m,0)}} \Big)^{1/q}}|L_{p}(\R^{n}, w_{\alpha})\Big\| \nonumber \\
&\sim \Big\|{\Big(\sum_{m \in \Z^{n-k}} \sum_{j \in \N_{0}} 2^{jsq} |{\lambda_{m}^{j,\ell}}|^{q} \chi_{E_{j,m}} \Big)^{1/q}}| {L_{p}(\R^{n}, w_{\alpha})}\Big\| \nonumber\\
&\sim \Big(\sum_{j \in \N_{0}} \sum_{m \in \Z^{n-k}} 2^{jsp} |{\lambda_{m}^{j, \ell}}|^{p} 2^{-j(n+\alpha)}\Big)^{1/p} \nonumber\\
&= \Big(\sum_{j \in \N_0} \sum_{m \in \Z^{n-k}} 2^{jp(s-\frac{n-k}{p}-\frac{\alpha+k}{p})} |{\lambda_{m}^{j, \ell}}|^{p} \Big)^{1/p}. 
\end{align}
And from here, we can prove
\begin{align*}
\norm{\ext f|F_{p,q}^{s}(\R^{n}, w_{\alpha})} & \lesssim \Big(\sum_{m \in \Z^{n-k}} |{\lambda_{m}}|^{p} \Big)^{1/p} + \sum_{\ell=2}^{2^{n-k}} \Big(\sum_{j \in \N_0} \sum_{m \in \Z^{n-k}} 2^{jp(s-\frac{n-k}{p}-\frac{\alpha+k}{p})} |{\lambda_{m}^{j,\ell}}|^{p} \Big)^{1/p} \nonumber \\
& \sim \norm{f|B_{p}^{s-\frac{\alpha+k}{p}}(\R^{n-k})} 
\end{align*}
\blue{in view of \eqref{ddh-1}.} Thus
\[
\ext: B_{p}^{s-\frac{\alpha+k}{p}}(\R^{n-k}) \hookrightarrow F_{p,q}^{s}(\R^{n}, w_{\alpha}) \quad \text{ is bounded,}
\]
and $\trk \circ \ext=\id: B_{p}^{s-\frac{\alpha+k}{p}}(\R^{n-k}) \rightarrow B_{p}^{s-\frac{\alpha+k}{p}}(\R^{n-k})$.
\end{proof}


For the particular case of weighted Sobolev spaces the previous result reads as follows. 
\begin{cor}\label{cor: trace of weighted sobolev spaces}
Let $1 < p < \infty$, $m \in \N$, $k\in \{1, \dots, n-1\}$ and $-1 < \alpha < p-1$ such that $ m >\frac{\alpha + k}{p}$. Then 
	\begin{enumerate}[\upshape\bfseries (i)]
		\item $\trk: W_{p}^{m}(\R^{n}, w_{\alpha}) \hookrightarrow B_{p}^{m-\frac{\alpha + k}{p}}(\R^{n-k})$ is \blue{linear and} bounded, and 
		\item there exists a linear and bounded extension operator
		\[
		\ext: B_{p}^{m-\frac{\alpha + k}{p}}(\R^{n-k}) \rightarrow W_{p}^{m}(\R^{n}, w_{\alpha})
		\]
		such that $\ \trk \circ \ext = \id: B_{p}^{m-\frac{\alpha + k}{p}}(\R^{n-k}) \rightarrow B_{p}^{m-\frac{\alpha + k}{p}}(\R^{n-k})$. 
	\end{enumerate}
	In particular, 
	\[\trk W_{p}^{m} (\R^{n}, w_{\alpha}) = B_{p}^{m-\frac{\alpha + k}{p}}(\R^{n-k}).\]
\end{cor}

\begin{rmk}
  Here we recover the result which has been  first obtained in \cite[Theorem~2.9.2]{T78} for $k=1$ and all spaces $W^m_p(\R^n, w_\alpha)$ with $1<p<\infty$ and $-1<\alpha<mp-1$, using quite different techniques.   
  It can be found as well in \cite[Proposition~4.10]{iwona-phd} in the context of traces on fractals. However, if $1<p<\infty$ and $p-1\leq \alpha<mp-1$, then it is not clear whether one still has the Littlewood-Paley assertion $W^m_p(\R^n, w_\alpha)=F^m_{p,2}(\R^n, w_\alpha)$, mentioned in Remark~\ref{rem-Triebel}.
  \end{rmk}

\section{Interpolation of Besov spaces}\label{interpol}
By a quasi-Banach couple $(A_1,A_2)$ we mean two quasi-Banach spaces $A_1$, $A_2$ which are continuously embedded in the same Hausdorff topological vector space $\mathcal{A}$. 

For a quasi-Banach couple $(A_1,A_2)$, $ 0 < \theta < 1$ and $ 0 < r \leq \infty$ the \emph{real interpolation space} $(A_1,A_2)_{\theta,r}$ consists of all $a \in A_1+A_2$ having finite quasi-norm
\[
\norm{a|(A_1,A_2)_{\theta,r}}= \Big( \int_{0}^{\infty} [t^{-\theta} K(t,a)]^{r} \frac{\dint t}{t} \Big)^{1/r},
\]
changing the integral by the supremum if $r=\infty$. Here, $K(t,a)$ is the \emph{Peetre's $K$-functional} defined by
\[
K(t,a) = K(t,a; A_1,A_2) = \inf\{ \norm{a_1|A_1}+t\norm{a_2|A_2}: a = a_1+a_2, \ a_j \in A_j\}, \quad t > 0, \ a \in A_1+A_2.
\]
If $t=1$, the $K$-functional $K(1,\cdot; A_1,A_2)$ coincides with the usual quasi-norm on $A_1+A_2$. The spaces $(A_1,A_2)_{\theta,q}$ are quasi-Banach spaces and they satisfy the \emph{interpolation property}, that is to say, if $(B_1,B_2)$ is another quasi-Banach couple and $T$ is a linear operator bounded from $A_j$ into $B_j$ for $j=1,2$, then $T$ is also bounded from $(A_1,A_2)_{\theta,r}$ into $(B_1,B_2)_{\theta,r}$. \magenta{See, for example, \cite{T78,BL} for more details about interpolation theory}.


Let $A$ be a quasi-Banach space, $(\Omega,\mu)$ a measure space, $0 < p_1 < p_2 < \infty$, $ 0 < r \leq \infty$ and $ 0 < \theta < 1$, then 
\begin{equation}\label{eq:interpolation formula Lorentz}
(L_{p_1}(A; \Omega,\mu), L_{p_2}(A; \Omega,\mu))_{\theta,r} = L_{p,r}(A; \Omega,\mu) \quad \text{with } \frac{1}{p} = \frac{1-\theta}{p_1}+\frac{\theta}{p_2},
\end{equation}
(see \cite[Section~1.18.6, Theorem~2 and Remark~5]{T78}). This formula is the key for proving that under the previous hypothesis
\begin{equation}\label{eq:interpolation formula Triebel-Lizorkin}
(F_{p_1,q}^{s}(\R^{n}), F_{p_2,q}^{\magenta{s}}(\R^{n}))_{\theta,r} = F_{q}^{s}L_{p,r}(\R^{n}), \quad s \in \R, \ 0 < q \leq \infty,
\end{equation}
where $F_{q}^{s}L_{p,r}(\R^{n})$ is the set of all $f\in \Sc'(\R^{n})$ with finite quasi-norm
\[
\norm{f|F_{q}^{s}L_{p,r}(\R^{n})} = \Big\|\Big(\sum_{j\in\N_0} 2^{jsq} \abs{(\varphi_{j}\hat{f})^{\vee}(\cdot)}^{q} \Big)^{1/q}|L_{p,r}(\R^{n}) \Big\|.
\]

Formula \eqref{eq:interpolation formula Triebel-Lizorkin} was proven by Triebel in \cite[Theorem 2.4.2/(c)]{T78} in the Banach case using the retract and coretract method and for the quasi-Banach case a new approach was given by Yang, Cheng and Peng \cite[Theorem 6]{YCP} based on the wavelet characterization of spaces $F_{q}^{s}L_{p,r}(\R^{n})$. In \cite[Chapter 4]{T92} and \cite[Section 2.2.1]{T20}, Triebel describes four key-problems for Triebel-Lizorkin and Besov spaces $A_{p,q}^{s}(\R^{n})$, $A \in \{B,F\}$: traces on hyperplanes, invariance with respect to diffeomorphisms of $\R^{n}$ onto itself, the existence of linear extension operators of the corresponding spaces $A_{p,q}^{s}(\R^{n}_{+})$ on $\R^{n}_{+}$ to $A_{p,q}^{s}(\R^{n})$ and several types of pointwise multipliers. \magenta{For $F_{q}^{s}L_{p,r}(\R^{n})$, t}hree of these key-problems can be treated satisfactorily  using the interpolation formula \eqref{eq:interpolation formula Triebel-Lizorkin} (see \cite{BCT}). However, this is not the case for the problem of traces as we are going to see now.

Let $ 0 < p < \infty$, $ 0 < q,r \leq \infty$, $s -\frac{1}{p} > (n-1)\Big(\frac{1}{p}-1\Big)_{+}$. Then, for $0 < p_1 < p < p_2 < \infty$ and $ 0 < \theta < 1$ such that $\frac{1}{p} = \frac{1-\theta}{p_1}+\frac{\theta}{p_2}$ and $s-\frac{1}{p_j} > (n-1)\Big(\frac{1}{p_j}-1 \Big)_{+}$, $j=1,2$, applying Theorem \ref{thm:trace 2} with $k=1$ and $\alpha =0$ we get that
\begin{align*}
\tr :F_{p_j,q}^{s}(\R^{n}) &\longrightarrow B_{p_j}^{s-\frac{1}{p_j}}(\R^{n-1}),
\end{align*}
for $j=1,2$ and using the interpolation property and \eqref{eq:interpolation formula Triebel-Lizorkin} we have
\[
\tr : F_{q}^{s}L_{p,r}(\R^{n}) \longrightarrow (B_{p_1}^{s-\frac{1}{p_1}}(\R^{n-1}), B_{p_2}^{s-\frac{1}{p_2}}(\R^{n-1}))_{\theta,r}.
\]
Analogously, $\ext:(B_{p_1}^{s-\frac{1}{p_1}}(\R^{n-1}), B_{p_2}^{s-\frac{1}{p_2}}(\R^{n-1}))_{\theta,r} \longrightarrow F_{q}^{s}L_{p,r}(\R^{n})$ and
\[\tr \circ \ext = \id: (B_{p_1}^{s-\frac{1}{p_1}}(\R^{n-1}), B_{p_2}^{s-\frac{1}{p_2}}(\R^{n-1}))_{\theta,r} \longrightarrow (B_{p_1}^{s-\frac{1}{p_1}}(\R^{n-1}), B_{p_2}^{s-\frac{1}{p_2}}(\R^{n-1}))_{\theta,r}.\] 
If $r = p$, then 
\[
(B_{p_1}^{s-\frac{1}{p_1}}(\R^{n-1}), B_{p_2}^{s-\frac{1}{p_2}}(\R^{n-1}))_{\theta,p} = B_{p}^{s-\frac{1}{p}}(\R^{n-1}),
\]
see \cite[Theorem 2.4.3]{T83}. But for $r \neq p$  the characterization of this interpolation space is an open problem already stated by Peetre in his monograph on Besov spaces \cite[p.110]{Peetre}. However, some of the ideas of the construction of the extension operator in Theorem \ref{thm:trace 2} will allow us to get a description of these spaces and, more generally, we are going to get a description of the interpolation space
\[
(B_{p_1}^{s_1-\frac{\alpha}{p_1}}(\R^{n}), B_{p_2}^{s-\frac{\alpha}{p_2}}(\R^{n}))_{\theta,r},
\]
for $0 < p_1 < p_2 < \infty$, $\alpha \in \R$, $ 0 < \theta < 1$ and $ 0 < r \leq \infty$. As a consequence, this gives a characterization of the trace of the Triebel-Lizorkin-Lorentz spaces $F_{q}^{s}L_{p,r}(\R^{n})$. 

We will need first some technical results. 
\begin{defi}\label{def:weighted T-L-L sequences}
Let $0 < p < \infty$, $ 0 < q,r \leq \infty$ and $\alpha > -1$. We define $\mathbf{f}_{q}L_{p,r}(\R^{n},w_{\alpha})$ as the set of all sequences $\lambda=(\lambda_{j,m})\subset \C$ such that
\[
\norm{\lambda|\mathbf{f}_{q}L_{p,r}(\R^{n},w_{\alpha})} := \Big\|\Big(\sum_{j\in\N_0} \sum_{m \in \Z^{n}} | \lambda_{j,m} \chi_{j,m} (\cdot)|^{q}\Big)^{1/q}|L_{p,r}(\R^{n},w_{\alpha}) \Big\| < \infty.
\]
If $p=r$, we use the notation $\mathbf{f}_{p,q}(\R^{n},w_{\alpha}) := \mathbf{f}_{q}L_{p,\magenta{p}}(\R^{n},w_{\alpha})=\mathbf{f}_q L_p(\R^n, w_\alpha)$. 
\end{defi}
\begin{rmk}\label{rmk:equivalent quasi-norm}
Under the hypothesis of Proposition \ref{prop:equivalent f norm}, we can also prove that
\[
\norm{\lambda|\mathbf{f}_{p,q}(\R^{n},w_{\alpha})} \sim \Big\|\Big(\sum_{j\in \N_{0}} \sum_{m \in \Z^{n}} \abs{\lambda_{j,m}}^{q} \chi_{E_{j,m}}(\cdot) \Big)^{1/q}|L_{p}(\R^{n}, w_{\alpha})\Big\|.
\]
\end{rmk}
\begin{prop}\label{prop:int weighted T-L sequences}
Let $0 < p_1 < p_2 < \infty$, $ 0 < \theta < 1$, $\frac{1}{p} = \frac{1-\theta}{p_1}+\frac{\theta}{p_2}$, $ 0 < q,r\leq \infty$, and $ -1 < \alpha$. Then
\[
(\mathbf{f}_{p_1,q}(\R^{n},w_\alpha), \mathbf{f}_{p_2,q}(\R^{n},w_{\alpha}))_{\theta,r} = \mathbf{f}_{q}L_{p,r}(\R^{n},w_{\alpha}),
\]
with equivalent quasi-norms. 
\end{prop}
\begin{proof}
For the case of unweighted spaces this result can be found in \cite[Theorem 3.5]{BCT}. Let $i=1,2$. Given any $\lambda \in \mathbf{f}_{p_i,q}(\R^{n},w_\alpha)$, let $R(\lambda)$ be the sequence of functions defined as
\[
R(\lambda)(x) = (\lambda_{j,m} \chi_{j,m}(x)).
\]
Then $R: \mathbf{f}_{p_i,q}(\R^{n},w_\alpha) \rightarrow L_{p_i}(\ell_{q}(\bl{q}), w_\alpha)$ is bounded and $\norm{R(\lambda)|L_{p_i}(\ell_{q}(\bl{q}), w_\alpha)}= \norm{\lambda|\mathbf{f}_{p_i,q}(\R^{n},w_\alpha)}$, $i=1,2$. 

Let $0 < A < \min( \frac{p_1}{r_{0}(w_\alpha)}, \frac{p_2}{r_{0}(w_\alpha)}, q)$. We consider now the operator 
\[P_{A}: L_{p_{i}}(\ell_{q}(\bl{q}), w_\alpha) \rightarrow \mathbf{f}_{p_i,q}(\R^{n},w_{\alpha})\]
 defined as
\[
P_{A}(g_{j,m}) := \Big( \Big(\frac{1}{\abs{Q_{j,m}}} \int_{Q_{j,m}} \abs{g_{j,m}}^{A} (y) \dint y \Big)^{1/A} \Big).
\]
Applying Lemma \ref{lem:maximal inequality 2}, one obtains
\begin{align*}
&\norm{P_{A} (g_{j,m})|\mathbf{f}_{p_i,q} (\R^{n},w_{\alpha})} \\
&\hspace{2cm}= \Big\| \Big(\sum_{\substack{j \in \N_{0},\\ m \in \Z^{n}}} \Big(\frac{1}{\abs{Q_{j,m}}} \int_{Q_{j,m}} \abs{g_{j,m}}^{A} (y) \dint y \Big)^{q/A} \chi_{j,m}(\cdot) \Big)^{A/q}|L_{p_{i}/A}(\R^{n}, w_{\alpha}) \Big\|^{1/A}\\
&\hspace{2cm}\leq \Big\|\Big(\sum_{\substack{j \in \N_0, \\ m \in \Z^{n}}} (M \abs{g_{j,m}}^{A})^{q/A} \Big)^{A/q}|L_{p_{i}/A} (\R^{n}, w_{\alpha}) \Big\|^{1/A}\\
&\hspace{2cm}\lesssim \Big\|\Big(\sum_{\substack{j \in \N_0, \\ m \in \Z^{n}}} ( \abs{g_{j,m}}^{A})^{q/A} \Big)^{A/q}|L_{p_{i}/A} (\R^{n}, w_{\alpha}) \Big\|^{1/A} \\
&\hspace{2cm} = \norm{(g_{j,m})|L_{p_i}(\ell_{q}(\bl{q}),w_\alpha)}.
\end{align*}
Note also that $P_{A} \circ R (\lambda) = \abs{\lambda}$, for every $\lambda \in \mathbf{f}_{p_i,q}(\R^{n},w_{\alpha})$. 

The linearity and boundedness of $R$ imply that for every $\lambda \in \mathbf{f}_{p_1,q}(\R^{n},w_{\alpha}) + \mathbf{f}_{p_2,q}(\R^{n},w_{\alpha})$ 
\begin{align}\label{eq:K-functional}
K(t, R(\lambda); L_{p_{1}}(\ell_{q}(\bl{q}), w_{\alpha}), L_{p_2}(\ell_{q}(\bl{q}), w_{\alpha})) \lesssim K(t, \lambda; \mathbf{f}_{p_1,q}(\R^{n},w_{\alpha}), \mathbf{f}_{p_2,q}(\R^{n},w_{\alpha})).
\end{align}
On the other hand, due to the lattice properties of $\mathbf{f}_{p_{i}, q}(\R^{n},w_{\alpha})$, we can rewrite its $K$-functional as
\begin{align*}
&K(t,\lambda; \mathbf{f}_{p_1,q}(\R^{n},w_{\alpha}), \mathbf{f}_{p_2,q}(\R^{n},w_{\alpha})) \\
& \hspace{1cm} = \inf \Big\lbrace \norm{\lambda^{1}|\mathbf{f}_{p_{1},q}(\R^{n},w_{\alpha})} + t \norm{\lambda^{2}|\mathbf{f}_{p_2,q}(\R^{n},w_{\alpha})}: \abs{\lambda_{j,m}} \leq \lambda_{j,m}^{1} + \lambda_{j,m}^{2},\lambda_{j,m}^{1},\lambda_{j,m}^{2} \geq 0 \Big\rbrace, 
\end{align*}
(see, for example, \cite[Lemma 3.1]{CM}). This together with the fact that $\abs{P_{A} g} \leq C_{A} (P_{A}(\abs{g_{1}}) + P_{A}(\abs{g_2}))$ if $g = g_1+g_2$, imply that
\begin{align*}
K(t, P_{A}g; \mathbf{f}_{p_1,q}(\R^{n},w_{\alpha}), \mathbf{f}_{p_2,q}(\R^{n},w_{\alpha})) &\lesssim \norm{P_{A}(\abs{g_1})|\mathbf{f}_{p_1,q}(\R^{n},w_{\alpha})}+ t \norm{P_{A}(\abs{g_{2}})|\mathbf{f}_{p_2,q}(\R^{n},w_{\alpha})} \\
&\lesssim \norm{g_{1}|L_{p_{1}}(\ell_{q}(\bl{q}), w_{\alpha})} + t \norm{g_{2}|L_{p_{2}}(\ell_{q}(\bl{q}), w_\alpha)}
\end{align*}
and taking the infimum over all possible decompositions we obtain that
\[
K(t, P_{A} g; \mathbf{f}_{p_1,q}(\R^{n},w_{\alpha}), \mathbf{f}_{p_2,q}(\R^{n},w_{\alpha})) \lesssim K(t, g; L_{p_1}(\ell_{q}(\bl{q}), w_{\alpha}), L_{p_{2}}(\ell_{q}(\bl{q}), w_{\alpha})) 
\]
for every $g = (g_{j,m}) \in L_{p_1}(\ell_{q}(\bl{q}), w_{\alpha})+ L_{p_{2}}(\ell_{q}(\bl{q}), w_{\alpha})$. Thus
\begin{align}\label{eq:K-functional 2}
K(t, \lambda; \mathbf{f}_{p_1,q}(\R^{n},w_{\alpha}), \mathbf{f}_{p_2,q}(\R^{n},w_\alpha)) &= K(t, \abs{\lambda}; \mathbf{f}_{p_1,q}(\R^{n},w_{\alpha}), \mathbf{f}_{p_2,q} (\R^{n},w_{\alpha})) \nonumber \\
&= K(t, P_{A} \circ R (\lambda); \mathbf{f}_{p_1,q}(\R^{n},w_{\alpha}), \mathbf{f}_{p_2,q}(\R^{n},w_\alpha)) \nonumber \\
&\lesssim K(t, R (\lambda); L_{p_1}(\ell_{q}(\bl{q}), w_{\alpha}), L_{p_{2}}(\ell_{q}(\bl{q}), w_{\alpha})).
\end{align}
From \eqref{eq:K-functional}, \eqref{eq:K-functional 2} and \eqref{eq:interpolation formula Lorentz}, we derive that
\begin{align*}
\norm{\lambda|(\mathbf{f}_{p_1,q}(\R^{n},w_\alpha), \mathbf{f}_{p_2,q}(\R^{n},w_\alpha))_{\theta, r}}&\sim \norm{R(\lambda)|(L_{p_1}(\ell_{q}(\bl{q}), w_{\alpha}), L_{p_{2}}(\ell_{q}(\bl{q}), w_{\alpha}))_{\theta, r}}\\
& = \norm{R(\lambda)|L_{p,r}(\ell_{q}(\bl{q}), w_{\alpha})} = \norm{\lambda|\mathbf{f}_{q}L_{p,r}(\R^{n},w_{\alpha})}.
\end{align*} 
\end{proof}
\begin{rmk}\label{rmk:equivalent quasi-norm Lorentz}
In the hypothesis of Definition \ref{def:weighted T-L-L sequences}, if $E_{j,m} \subset Q_{j,m}$ and $\abs{E_{j,m}} \sim \abs{Q_{j,m}}$, $j \in \N_0$ and $m \in \Z^{n}$, then applying Remark \ref{rmk:equivalent quasi-norm} and Proposition \ref{prop:int weighted T-L sequences}, we have the following equivalence of quasi-norms: 
\[
\norm{\magenta{\lambda}|\mathbf{f}_{q}L_{p,r}(\R^{n},w_{\alpha})} \sim \Big\|\Big(\sum_{j\in \N_0} \sum_{m \in \Z^{n}} \abs{\lambda_{j,m}}^{q} \chi_{E_{j,m}} (\cdot) \Big)^{1/q} | {L_{p,r}(\R^{n}, w_{\alpha})}\Big\|,
\] 
for every $\lambda \in \mathbf{f}_{q}L_{p,r}(\R^{n},w_{\alpha})$. 
\end{rmk}
Let $k \in \N$ and let $(A_\ell)_{\ell=1}^{k}$ be a finite sequence of quasi-Banach spaces. We consider the space $\prod_{\ell=1}^{k} A_{\ell} = A_{1} \times \dots \times A_{k}$ quasi-normed by
\[
\Big\|(a_1,\dots,a_k)|\prod_{\ell=1}^{k} A_{\ell}\Big\| = \sum_{\ell=1}^{k} \norm{a_{\ell}|A_\ell}, \quad a_{\ell} \in A_{\ell},\  \ell=1,\dots,k.
\]


\begin{prop}\label{prop:interpolation of products}
Let $((A^1_{\ell}, A^2_{\ell}))_{\ell=1}^{k}$ be a finite sequence of quasi-Banach couples and put $B_{i} = \prod_{\ell=1}^{k} A^i_{\ell}$, $i=1,2$. Then $(B_1, B_2)$ is a quasi-Banach couple and for any $ 0 < \theta < 1$ and $ 0 < r \leq \infty$, 
\[
(B_1, B_2)_{\theta, r} = \prod_{\ell=1}^{k} (A^1_{\ell}, A^2_{\ell})_{\theta,r}, 
\]
with equivalent quasi-norms. 
\end{prop}
\begin{proof}
First note that 
\[
B_{i} \hookrightarrow \prod_{\ell=1}^{k} (A^1_{\ell} + A^2_{\ell}), \text{ for } i=1,2.
\]
Indeed, let $a =(a_1,\dots,a_k) \in B_i$, then 
\begin{align*}
\norm{a|\prod_{\ell=1}^{k} (A^1_{\ell} + A^2_{\ell})} = \sum_{\ell=1}^{k} \norm{a_{\ell}|A^1_{\ell}+A^2_{\ell}} \leq \sum_{\ell=1}^{k} \norm{a_{\ell}|A^i_{\ell}} = \norm{a|B_{i}}.
\end{align*}
Now we prove that $(B_1, B_2)_{\theta,\magenta{r}} \hookrightarrow \prod_{\ell=1}^{k} (A^1_{\ell}, A^2_{\ell})_{\theta,\magenta{r}}$ . Fix $t> 0$ and let $a = (a_1, \dots ,a_{k}) \in B_1 + B_2$. For every $\eps > 0$ we can decompose $a = a^{1}+a^{2}$ with $a^{i} = (a_{1}^{i},\dots,a_{k}^{i}) \in B_{i}$, $i=1,2$ and such that 
\[
\norm{a^{1}|{B_1}}+ t \norm{a^2|{B_{2}}} = \sum_{\ell=1}^{k} \Big(\norm{a_{\ell}^{1}|A^1_{\ell}}+ t \norm{a_{\ell}^{2}|A^2_{\ell}}\Big) \leq K(t,a; B_1, B_2) + \eps.
\]
Then, $\sum_{\ell=1}^{k} K(t, a_{\ell}; A^1_{\ell}, A^2_{\ell}) \leq \sum_{\ell=1}^{k} (\norm{a_{\ell}^{1}|A^1_{\ell}} + t \norm{a_{\ell}^{2}|A^2_{\ell}}) \leq K(t,a; B_1, B_2) + \eps$, for every $\eps > 0$. From here we deduce that
\begin{align*}
\Big\|{a|\prod_{\ell=1}^{k} (A^1_{\ell}, A^2_{\ell})_{\theta, r}}\Big\| &= \sum_{\ell=1}^{k} \norm{a_{\ell}|(A^1_{\ell}, A^2_{\ell})_{\theta,r}} \sim \Big(\int_{0}^{\infty}t^{-\theta r} \Big(\sum_{\ell=1}^{k} K(t,a_\ell; A^1_{\ell}, A^2_{\ell}) \Big)^{r} \frac{\dint t}{t}\Big)^{1/r}\\
&\leq \left(\int_{0}^{\infty} t^{-\theta r} K(t,a; B_1, B_2)^{r} \frac{\dint t}{t} \right)^{1/r} = \norm{a|(B_1, B_2)_{\theta,r}}.
\end{align*}
An analogous argument works for the other direction. 
\end{proof}
Let $E_{j,m}=Q_{j,m}^{(n-k)} \times (2^{-j-2}, 2^{-j-1})^{k}$ be as in the proof of Theorem \ref{thm:trace 2}.
\begin{prop}\label{prop:interpolation besov}
Let $ 0 < p_1 < p_2 < \infty$, $\alpha > -1$, $n \in \N$, $k=1,2,\dots,n-1$, $s-\frac{\alpha+k}{p_{i}} > (n-k)(\frac{1}{p_{i}}-1)_{+}$ for $i=1,2$, $ 0 < \theta < 1$ and $ 0 < r \leq \infty$. Then 
\[
(B_{p_1}^{s-\frac{\alpha + k}{p_{1}}}(\R^{n-k}), B_{p_{2}}^{s-\frac{\alpha + k}{p_2}}(\R^{n-k}))_{\theta, r}
\]
is the set of all $f \in \Sc'(\R^{n-k})$ that admit a wavelet representation on $\R^{n-k}$
\begin{align}\label{eq:wavelet repres. in R(n-k)}
f(x) = \sum_{m \in \Z^{n-k}} \lambda_{m} \psi_{m} (x) + \sum_{\ell=2}^{2^{n-k}} \sum_{j \in \N_0} \sum_{m \in \Z^{n-k}} \lambda_{m}^{j, \ell} 2^{-j(n-k)/2} \psi_{\ell,m}^{j}(x) \quad \text{for every } x \in \R^{n-k}
\end{align}
and 
\[\norm{f|(B_{p_1}^{s-\frac{\alpha + k}{p_{1}}}(\R^{n-k}), B_{p_{2}}^{s-\frac{\alpha + k}{p_2}}(\R^{n-k}))_{\theta, r}} \sim \norm{\Lambda^{s}(\lambda)|L_{p,r}(\R^{n}, w_{\alpha})} < \infty,\]
where
\[
\Lambda^{s}(\lambda)(x) = \sum_{m \in \Z^{n-k}} \abs{\lambda_{m}} \chi_{E_{0,m}}(x) + \sum_{\ell=2}^{2^{n-k}} \sum_{m \in \Z^{n-k}} \sum_{j=0}^{\infty} 2^{js} |\lambda_{m}^{j, \ell}|\chi_{E_{j,m}} (x), \ x \in \R^{n}.
\]
\end{prop}
\begin{proof}
By Theorem \ref{thm:wavelet Besov spaces}, we know that $(B_{p_1}^{s-\frac{\alpha + k}{p_{1}}}(\R^{n-k}), B_{p_{2}}^{s-\frac{\alpha + k}{p_2}}(\R^{n-k}))_{\theta, r}$ is the set of all $f \in \Sc'(\R^{n-k})$ that can be represented as in \eqref{eq:wavelet repres. in R(n-k)} and 
\[
\|f|(B_{p_1}^{s-\frac{\alpha + k}{p_{1}}}(\R^{n-k}), B_{p_{2}}^{s-\frac{\alpha + k}{p_2}}(\R^{n-k}))_{\theta, r}\|\sim \norm{\lambda|(b_{p_1}^{s-\frac{\alpha+k}{p_1}}(\R^{n-k}), b_{p_2}^{s-\frac{\alpha+k}{p_2}}(\R^{n-k}))_{\theta,r}}.
\]
Let $i=1,2$. Consider the operator $T: b_{p_i}^{s-\frac{\alpha + k}{p_{i}}}(\R^{n-k}) \rightarrow \prod_{\ell=1}^{2^{n-k}} \mathbf{f}_{p_i 1}(\R^{n}, w_{\alpha})$ defined as
\[
T((\lambda_{m}), (\lambda_{m}^{j,2}),\dots,(\lambda_{m}^{j,2^{n-k}})) = (\hat{\lambda}_{j,(m,M)}^{\ell})= \begin{cases} 
\lambda_{m} &\text{ if } \ell = 1, j =0 \text{ and } M=0, \\
2^{js} \lambda_{m}^{j, \ell} &\text{ if } \ell=2,\dots,2^{n-k} \text{ and } M=0, \\
0 &\text{ otherwise,}
\end{cases}
\]
with $j \in \N_{0}$, $(m,M) \in \R^{n-k}\times \R^{k}$ and $\ell=1,2,\dots,2^{n-k}$. From \eqref{eq:estimate g1} and \eqref{eq:estimate gl} with $q=1$, we deduce that $T$ is bounded. Now we consider the operator
\begin{align*}
P: \prod_{\ell=1}^{2^{n-k}} \mathbf{f}_{p_i,1}(\R^{n}, w_{\alpha}) &\longrightarrow b_{p_{i}}^{s-\frac{\alpha + k}{p_{i}}}(\R^{n-k})\\
(\lambda_{j, (m,M)}^{\ell} ) &\longrightarrow ((\lambda_{0,(m,0)}), (2^{-js}\lambda_{j,(m,0)}^{1}),\dots,(2^{-js}\lambda_{j,(m,0)}^{2^{n-k}})).
\end{align*}
Observe that $P \circ T = \id: b_{p_{i}}^{s-\frac{\alpha + k}{p_{i}}}(\R^{n-k}) \rightarrow b_{p_{i}}^{s-\frac{\alpha + k}{p_{i}}}(\R^{n-k})$, $i=1,2$ and using Remark \ref{rmk:equivalent quasi-norm}, 
\begin{align*}
&\norm{P(\lambda)|b_{p_{i}}^{s-\frac{\alpha + k}{p_{i}}}(\R^{n-k})} = \Big( \sum_{m \in \Z^{n}} |\lambda_{0,(m,0)}^{1}|^{p_i} \Big)^{1/p_{i}}+ \sum_{\ell=2}^{2^{n-k}} \Big( \sum_{j=0}^{\infty} \sum_{m \in \Z^{n}} |\lambda_{j,(m,1)}^{\ell}|^{p_{i}} 2^{-j(\alpha +k)}\Big)^{1/p_{i}}\\
&\ \ \ \ \ \ \ \sim \Big\|\sum_{m \in \Z^{n}} |{\lambda_{0,(m,0)}^{1}}| \chi_{E_{0,(m,0)}} | \ L_{p_{i}}(\R^{n}, w_{\alpha})\Big\| + \sum_{\ell=2}^{2^{n-k}} \Big\|\sum_{j=0}^{\infty} \sum_{m \in \Z^{n}} |\lambda_{j,(m,0)}^{\ell}|\chi_{E_{j,(m,0)}} | \ L_{p_i}(\R^{n}, w_{\alpha})\Big\|\\
&\ \ \ \ \ \ \ \ \sim \norm{\lambda | \prod_{\ell=1}^{2^{n-k}} \mathbf{f}_{p_{i},1}(\R^{n},w_{\alpha})}.
\end{align*}
Then, applying Proposition \ref{prop:interpolation of products}, Proposition \ref{prop:int weighted T-L sequences} and Remark \ref{rmk:equivalent quasi-norm Lorentz}, we have that
\begin{align*}
&\norm{\lambda|(b_{p_1}^{s-\frac{\alpha+k}{p_1}}(\R^{n-k}), b_{p_2}^{s-\frac{\alpha+k}{p_2}}(\R^{n-k}))_{\theta,r}} \sim \norm{T(\lambda)|(\prod_{\ell=1}^{2^{n-k}} \mathbf{f}_{p_{1},1}(\R^{n},w_{\alpha}), \prod_{\ell=1}^{2^{n-k}} \mathbf{f}_{p_{2},1}(\R^{n},w_{\alpha}) )_{\theta,r}} \\
&\ \ \ \ \ \ \sim \norm{T(\lambda)|\prod_{\ell=1}^{2^{n-k}}\mathbf{f}_{1}L_{p,r}(\R^{n},w_{\alpha})}\sim \norm{\Lambda^{s}(f)|L_{p,r}(\R^{n}, w_{\alpha})}. 
\end{align*}
\end{proof}
\begin{cor}\label{cor:interpolation besov spaces}
Let $0 < p_1 < p_2 < \infty$, $\alpha > 0$, $n \in \N$, $ 0 < \theta < 1$, $0 < r \leq \infty$ and $ s \in \R$. Then
\[
(B_{p_1}^{s-\frac{\alpha}{p_1}}(\R^{n}), B_{p_2}^{s-\frac{\alpha}{p_2}}(\R^{n}))_{\theta,r}
\]
is the set of all $f \in \Sc'(\R^{n})$ that admit a wavelet representation on $\R^{n}$
\begin{align}\label{eq: wavelet representation}
f(x) = \sum_{m \in \Z^{n}} \lambda_{m} \psi_{m}(x) + \sum_{\ell=2}^{2^{n}} \sum_{j=0}^{\infty} \sum_{m \in \Z^{n}} \lambda_{m}^{j,\ell} 2^{-jn/2}\psi_{\ell,m}^{j}(x), \quad x \in \R^{n}, 
\end{align}
and $\norm{f|(B_{p_1}^{s-\alpha/p_1}(\R^{n}), B_{p_2}^{s-\alpha/p_2}(\R^{n}))_{\theta,r}} \sim \norm{\Lambda^{s}(\lambda)|L_{p,r}(\R^{n+1}, w_{\alpha-1})} < \infty$, where
\[
\Lambda^{s}(\lambda)(x) =  \sum_{m \in \Z^{n}} \abs{\lambda_m} \chi_{E_{0,m}}(x) + \sum_{\ell=2}^{2^{n}} \sum_{m \in \Z^{n}} \sum_{j=0}^{\infty} 2^{js} |\lambda_{m}^{j,\ell}| \chi_{E_{jm}}(x), \quad x \in \R^{n+1}.
\]
Now $E_{j,m} = Q_{j,m}^{(n)} \times (2^{-j-2}, 2^{-j-1}) \subset \R^{n+1}$.
\end{cor}

\begin{proof}
By Theorem \ref{thm:wavelet Besov spaces}, we know that $(B_{p_1}^{s-\frac{\alpha}{p_1}}(\R^{n}), B_{p_2}^{s-\frac{\alpha}{p_2}}(\R^{n}))_{\theta,r}$ is the set of all $f \in \Sc'(\R^{n})$ that can be represented as in \eqref{eq: wavelet representation} and 
\[
\norm{f|(B_{p_1}^{s-\frac{\alpha}{p_1}}(\R^{n}), B_{p_2}^{s-\frac{\alpha}{p_2}}(\R^{n}))_{\theta,r}} \sim \norm{\lambda|(b_{p_1}^{s-\frac{\alpha}{p_1}}(\R^{n}), b_{p_2}^{s-\frac{\alpha}{p_{2}}}(\R^{n}))_{\theta,r}}.
\]
Let $\tilde{s} > \max \Big\lbrace \frac{\alpha}{p_1}+n\Big(\frac{1}{p_1}-1\Big)_{+}, \frac{\alpha}{p_2}+n\Big(\frac{1}{p_2}-1\Big)_{+}\Big\rbrace  \blue{= \frac{\alpha}{p_1}+n\Big(\frac{1}{p_1}-1\Big)_{+}}$.
Then
\begin{align*}
I: b_{p_{i}}^{s-\frac{\alpha}{p_i}}(\R^{n}) &\longrightarrow b_{p_{i}}^{\tilde{s}-\frac{\alpha}{p_i}}(\R^{n})\\
\lambda=((\lambda_{m}), (\lambda_{m}^{j,1}),\dots,(\lambda_{m}^{j,2^{n}}))&\longrightarrow ( (\lambda_{m}), (2^{j(s-\tilde{s})}\lambda_{m}^{j,1}),\dots,(2^{j(s-\tilde{s})}\lambda_{m}^{j,2^{n}})),
\end{align*}
is an isometry for $i=1,2$ and applying Proposition \ref{prop:interpolation besov} we get
\begin{align*}
\norm{\lambda|(b_{p_1}^{s-\frac{\alpha}{p_1}}(\R^{n}), b_{p_2}^{s-\frac{\alpha}{p_{2}}}(\R^{n}))_{\theta,r}} &\sim \norm{I(\lambda)|(b_{p_1}^{\tilde{s}-\frac{\alpha}{p_1}}(\R^{n}), b_{p_2}^{\tilde{s}-\frac{\alpha}{p_{2}}}(\R^{n}))_{\theta,r}}\\
&\sim \norm{\Lambda^{\tilde{s}}(I\lambda)|L_{p,r}(\R^{n+1},w_{\alpha-1})}\\
& = \norm{\Lambda^{s}(\lambda)|L_{p,r}(\R^{n+1},w_{\alpha-1})}.
\end{align*}
\end{proof}
\begin{rmk}\label{rmk:interpolation besov spaces}
Using directly results of interpolation theory, we can also obtain Corollary \ref{cor:interpolation besov spaces}. Assume first that $s=0$. Observe that $b_{p_i}^{-\alpha/p_{i}}(\R^{n}) = \ell_{p_i} \times \prod_{\ell=2}^{2^{n}} \ell_{p_i}^{-\frac{\alpha + n}{p_i}}(\bl{p_i})$ for $i=1,2$ and notice also that
\begin{align*}
\ell_{p_i} &= L_{p_i}(\Z^{n},\mu) &\quad \text{with }& &\mu = \sum_{m \in \Z^{n}} \delta_{\{m\}}, \\
\ell_{p_i}^{-\frac{(\alpha + n)}{p_i}}(\bl{p_i}) &= L_{p_i}(\N_{0} \times \Z^{n}, \hat{\mu}) &\quad \text{with }& &\hat{\mu} = \sum_{j=0}^{\infty} \sum_{m\in \Z^{n}} 2^{-j(\alpha + n)} \delta_{\{(j,m)\}}.
\end{align*}
By Proposition \ref{prop:interpolation of products}, we have that
\[
(b_{p_1}^{-\alpha/p_1}(\R^{n}), b_{p_2}^{-\alpha/p_2}(\R^{n}))_{\theta,r} = L_{p,r}(\Z^{n},\mu) \times \prod_{\ell=2}^{2^{n}} L_{p,r}(\N_0 \times \Z^{n}, \hat{\mu}).
\]
Now for any $s \in \R$,
\begin{align*}
I_{s}: b_{p_{i}}^{s-\frac{\alpha}{p_i}}(\R^{n}) &\longrightarrow b_{p_{i}}^{-\frac{\alpha}{p_i}}(\R^{n}) \\
\lambda=((\lambda_{m}), (\lambda_{m}^{j,1}),\dots,(\lambda_{m}^{j,2^{n}}))&\longrightarrow ( (\lambda_{m}), (2^{js}\lambda_{m}^{j,1}),\dots,(2^{js}\lambda_{m}^{j,2^{n}})),
\end{align*}
is an isometry for $i=1,2$ and using Theorem \ref{thm:wavelet Besov spaces} we get
\begin{align*}
&\norm{f|(B_{p_1}^{s-\frac{\alpha}{p_1}}(\R^{n}), B_{p_2}^{s-\frac{\alpha}{p_2}}(\R^{n}))_{\theta,r}} \\
&\hspace{2.5cm}\sim \norm{\lambda|(b_{p_1}^{s-\alpha/p_1}(\R^{n}), b_{p_2}^{s-\alpha/p_2}(\R^{n}))_{\theta,r}} \\
&\hspace{2.5cm} \sim \norm{I_{s}(\lambda)|(b_{p_1}^{-\alpha/p_1}(\R^{n}), b_{p_2}^{-\alpha/p_2}(\R^{n}))_{\theta,r}}\\
&\hspace{2.5cm} \sim \norm{(\lambda_{m})|L_{p,r}(\Z^{n},\mu)} + \sum_{\ell=2}^{2^{n}} \norm{(2^{js}\lambda_{m}^{j,\ell})|L_{p,r}(\N_0 \times \Z^{n}, \hat{\mu})}\\
&\hspace{2.5cm} = \norm{t^{1-1/r}\text{card}\{m \in \Z^{n}: \abs{\lambda_{m}} > t\}^{1/p}|L_{r}(\R)}\\
& \hspace{3.5cm}+ \sum_{\ell=2}^{2^{n}} \Big\|{t^{1-1/r} \Big(\sum_{j=0}^{\infty} 2^{-j(n+\alpha)} \text{card}\{m \in \Z^{n}: |2^{js}\lambda_{m}^{j,\ell}| > t\}\Big)^{1/p}|L_{r}(\R)}\Big\|.
\end{align*}
This quasi-norm is equivalent to the one given in Corollary \ref{cor:interpolation besov spaces}. Indeed, as $E_{j,m}$ in Corollary \ref{cor:interpolation besov spaces} are disjoint sets with $\int_{\R^{n+1}} \chi_{E_{j,m}}(x) w_{\alpha-1}(x) \dint x \sim 2^{-j(\alpha+n)}$, then for $\dint\tilde{\mu} = w_{\alpha-1} \dint x$ in $\R^{n+1}$, it is clear that
\begin{align*}
&\tilde{\mu}\left(\{ x \in \R^{n+1}: \abs{\Lambda^{s}(f)}>t\}\right)\\
& \ \ \ \ \ \ \ \ \ \ = \text{card}\{m \in \Z^{n}: \abs{\lambda_{m}} > t\}+\sum_{\ell=2}^{2^{n}} \sum_{j=0}^{\infty} 2^{-j(n+\alpha)} \text{card}\{m \in \Z^{n}: |2^{js}\lambda_{m}^{j,\ell}| > t\}
\end{align*}
and 
\begin{align*}
&\norm{\Lambda^{s}(\lambda)|L_{p,r}(\R^{n+1}, w_{\alpha-1})} \sim \norm{t^{1-1/r}\text{card}\{m \in \Z^{n}: \abs{\lambda_{m}} > t\}^{1/p}|L_{r}(\R)}\\
& \ \ \ \ \ \ \ \ \ + \sum_{\ell=2}^{2^{n}} \norm{t^{1-1/r} \Big(\sum_{j=0}^{\infty} 2^{-j(n+\alpha)} \text{card}\{m \in \Z^{n}: |2^{js}\lambda_{m}^{j,\ell}| > t\}\Big)^{1/p}|L_{r}(\R)}.
\end{align*}

Finally, notice that with the interpolation approach, the argument works for any $\alpha \in \R$, we do not need to assume that $\alpha >0$. \blue{One might understand this, at first glance, astonishing fact in the way that for smaller $\alpha$ the corresponding smoothness of the involved Besov spaces $B^{s-\alpha/p_i}_{p_i}(\R^n)$, $i=1,2$, increases such that the Besov spaces themselves are getting smaller. Accordingly the assumptions to the sequences representing them become more and more restrictive as can be observed from the argument including the cardinality above.}
\end{rmk}



\bigskip
\noindent \textbf{Acknowledgements.}

\noindent \small{Part of the research of Blanca F. Besoy was done while she visited the Institute of Mathematics of the Friedrich Schiller University Jena. She would like to thank the Research Group on Function Spaces of this Institute for the hospitality given to her during the visit.}

\end{document}